\documentclass[aos,seceqn,nameyear,dvips]{arximspdf}
\usepackage{graphicx}

\doi{10.1214/10-AOS840}
\volume{38}
\issue{6}
\pubyear{2010}
\firstpage{3751}
\lastpage{3781}

\makeatletter

\newcommand{\eqref}[1]{(\ref{#1})}

\newcommand{\E}{\mathrm{E}}
\newcommand{\dom}{\mathop{\operatorname{dom}}}
\newcommand{\argmin}{\mathop{\arg\min}}
\newcommand{\argmax}{\mathop{\arg\max}}
\newcommand{\ev}{\mathrm{ev}}
\newcommand{\cl}{\mathrm{cl}}
\newcommand{\ri}{\mathrm{ri}}
\newcommand{\lev}{\operatorname{lev}}
\newcommand{\curv}{\operatorname{curv}}
\newcommand{\esssup}{\mathop{\operatorname{ess}\sup}}
\newcommand{\supp}{\operatorname{supp}}

\newcommand{\Real}{\mathbb{R}}
\newcommand{\Realc}{\overline{\mathbb{R}}}
\newcommand{\RealP}{\mathbb{R}_{+}}
\newcommand{\RealPc}{{\overline{\mathbb{R}}}_{+}}
\newcommand{\Conv}{\mathcal{C}}
\newcommand{\Model}{\mathcal{P}}
\newcommand{\ModelC}{\mathcal{G}}
\newcommand{\PP}{\mathbb{P}}

\newcommand{\pnorm}[1]{V({#1})}

\newproclaim{AssumD}{}

\newtheorem{theorem}{Theorem}[section]
\newtheorem{lem}[theorem]{Lemma}
\newtheorem{cor}[theorem]{Corollary}
\newproclaim{defn}[theorem]{Definition}
\newproclaim{rem}[theorem]{Remark}
\newproclaim{ex}[theorem]{Example}

\makeatother

\begin{document}
\begin{frontmatter}

\title{Nonparametric estimation of multivariate convex-transformed densities}
\runtitle{Convex-transformed densities}

\begin{aug}
\author[A]{\fnms{Arseni} \snm{Seregin}\corref{}\thanksref{t1}\ead[label=e1]{arseni@stat.washington.edu}} and
\author[A]{\fnms{Jon A.} \snm{Wellner}\thanksref{t1,t2}\ead[label=e2]{jaw@stat.washington.edu}}
\runauthor{A. Seregin and J. A. Wellner}
\affiliation{University of Washington}
\address[A]{Department of Statistics\\
University of Washington\\
Box 354322\\
Seattle, Washington 98195-4322\\
USA\\
\printead{e1}\\
\phantom{E-mail: }\printead*{e2}}
\end{aug}

\thankstext{t1}{Supported in part by NSF Grant DMS-08-04587.}
\thankstext{t2}{Supported in part by NI-AID Grant 2R01 AI291968-04.}

% HISTORY:
\received{\smonth{1} \syear{2010}}
\revised{\smonth{5} \syear{2010}}

% ABSTRACT
%
\begin{abstract}
We study estimation of multivariate densities $p$ of the form $p(x) =
h(g(x))$ for $x \in\Real^d$
and for a fixed monotone function $h$ and an unknown convex function $g$.
The canonical example is $h(y) = e^{-y}$ for $y \in\Real$; in this
case, the resulting class of
densities
\[
\Model(e^{-y}) = \{ p = \exp(-g) \dvtx  g  \mbox{ is convex} \}
\]
is well known as the class
of \textit{log-concave} densities. Other functions $h$ allow for classes
of densities with
heavier tails than the log-concave class.

We first investigate when the maximum likelihood estimator $\hat
{p}$ exists for the class
$\Model(h)$ for various choices of monotone transformations $h$,
including decreasing and
increasing functions $h$.
The resulting models for increasing transformations $h$ extend the
classes of \textit{log-convex} densities
studied previously in the econometrics literature, corresponding to
$h(y) = \exp(y)$.

We then establish consistency of the maximum likelihood estimator for
fairly general functions $h$, including the
log-concave class $\Model(e^{-y})$ and many others.
In a final section, we provide asymptotic minimax lower
bounds for the estimation of $p$ and its vector of derivatives at a
fixed point $x_0$ under
natural smoothness hypotheses on $h$ and $g$.
The proofs rely heavily on results from convex analysis.
\end{abstract}

% KEYWORDS
%

\end{frontmatter}

%----------------------------------------------------------
%s1 ###
\section{Introduction and background}
\label{sec:intro}

%----------------------------------------------------------

%s1.1 ###
\subsection{Log-concave and $r$-concave densities}
\label{sec:logconcaveAndrconcavedensities}
A probability density $p$ on $\Real^d$ is called \textit{log-concave}
if it
can be written as
\[
p(x) = \exp( -g(x))
\]
for some convex function
$g\dvtx \Real^d \to(-\infty,\infty]$.
We let $\Model(e^{-y})$ denote the class of all log-concave densities
on $\Real^d$.
As shown by
\citet{MR0087249}, a density function $p$ on $\Real$
is log-concave if and only if its convolution with any unimodal density
is again unimodal.

Log-concave densities have proven to be useful in a wide range of
statistical problems;
see \citet{Walther2009}
for a survey of recent developments and statistical applications of
log-concave densities
on $\Real$ and $\Real^d$, and see
\citet{culeSamworthStewart2010} for several interesting
applications of estimators of such densities in $\Real^d$.

Because the class of multivariate log-concave densities contains the
class of multivariate normal densities
and is preserved under a number of important operations (such as
convolution and marginalization), it serves as
a valuable nonparametric surrogate or replacement for the class of
normal densities.
Further study of the class
of log-concave densities from this perspective has been undertaken by
\citet{SchuhmacherHuslerDuembgen2009}.
%On the analysis side, various isoperimetric and Poincar\'e type
%inequalities have been obtained by
%and \citet{MR2386937}. %Milmann and Soshin

Log-concave densities have the slight drawback that the tails must be
decreasing exponentially, so
a number of authors, including
\citet{mizera2008}, have proposed using generalizations of the
log-concave family involving
$r$-concave densities, defined as follows. For $a,b \in\Real$, $r \in
\Real$ and $\lambda\in(0,1)$, define
the generalized mean of order $r$, $M_r (a,b; \lambda)$, for $a,b \ge
0$, by
\[
M_r (a,b; \lambda)
= \cases{
\bigl((1-\lambda) a^r + \lambda b^r \bigr)^{1/r}, &\quad $r \not=
0,  a,b>0$, \cr
0, &\quad $r < 0,  ab = 0$, \cr
a^{1-\lambda} b^{\lambda} , &\quad $r = 0$.}
\]
A density function $p$ is then $r$-concave on $C \subset\Real^d$ if
and only
\[
p\bigl( (1-\lambda) x + \lambda y\bigr) \ge M_r (p(x) , p(y) ;
\lambda)\qquad
\mbox{for all }  x,y \in C,  \lambda\in(0,1).
\]
We denote the class of all $r$-concave densities on $C \subset\Real
^d$ by $\widehat{\Model}(y_{+}^{1/r}; C)$ and write
$\widehat{\Model}(y_{+}^{1/r}) $ when $C = \Real^d$.
As noted by Dharmadhikari and Joag-Dev [(\citeyear{MR954608}),
page~86], for $r \le0$, it suffices to consider $\widehat{\Model
}(y_{+}^{1/r})$, and
it is almost immediate from the definitions that $p \in\widehat
{\Model}(y_{+}^{1/r})$ if and only if
$p(x) = (g(x))^{1/r}$ for some convex function $g$ from $\Real^d$ to
$[0,\infty)$.
For $r>0$, $p \in\widehat{\Model}(y_{+}^{1/r}; C)$ if and only if $
p(x) = (g(x))^{1/r}$, where $g$ mapping $C$ into $(0,\infty)$ is
concave.

These results motivate definitions of the classes
$\Model(y_{+}^{-s}) = \{ p(x) = g(x)^{-s} \dvtx  g$ is convex$\}
$ for $s \ge0$
and, more generally, for a fixed monotone function $h$ from $\Real$ to
$\Real$,
\[
\Model(h) \equiv\{ h \circ g \dvtx  g  \mbox{ convex} \}.
\]
Such generalizations of log-concave densities and log-concave measures
based on means of order $r$
have been introduced by a series of authors, sometimes with differing
terminology, apparently starting
with
\citet{MR0301151}, %AvrielMR0301151
and continuing with\vadjust{\goodbreak}
\citet{MR0404559}, %Borell
\citet{MR0450480}, %Brascamp-Lieb
\citet{MR0404557}, %Prekopa
\citet{MR0428540} %Rinott
and
\citet{MR735860}. %Uhrin
A nice summary of these connections is given by \citet{MR954608}.
These authors also present results concerning the preservation of
$r$-concavity under
a variety of operations, including products, convolutions and marginalization.
%In the mathematics literature, the underlying fundamental inequality
%has come to be known as the
%Borell-Brascamp-Lieb inequality; see e.g. \citet{MR1865396}.
%For these heavy-tailed classes, development of isoperimetric and
%Poincar\'e inequalities is
%also underway: see e.g. \citet{MR2336600} and

Despite the longstanding and current rapid development of the properties
of such classes of densities on the probability side, very little
has been done from the standpoint of nonparametric estimation,
especially when $d \ge2$.
%For $d=1$, ....

Nonparametric estimation of a log-concave density on $\Real^d$ was
initiated by
\citet{culeSamworthStewart2010}.
These authors developed an algorithm for computing their estimators
and explored several interesting applications.
\citet{mizera2008} developed a family of penalized criterion
functions related to the R\'enyi divergence measures
and explored duality in the optimization problems. They did not succeed
in establishing consistency of
their estimators, but did investigate Fisher consistency.
Recently, \citet{Samworth2010} have established consistency of
the (nonparametric) maximum likelihood
estimator of a log-concave density on $\Real^d$, even in a setting of
model misspecification: when the true density is not log-concave, the
estimator converges to the closest log-concave density to the true
density, in the sense of Kullback--Leibler
divergence.

In this paper, our goal is to investigate maximum likelihood estimation
in the classes $\Model(h)$
corresponding to a fixed monotone (decreasing or increasing) function $h$.
In particular, for decreasing functions $h$, we handle all of the
$r$-concave classes
$\Model(y_{+}^{1/r} )$
%$\Model_r$
with $r = -1/s$ and $r \le-1/d$ (or $s \ge d$).
On the increasing side, we treat, in particular,
the cases $h(y) = y 1_{[0,\infty)} (y)$ and $h(y) = e^y$ with $C =
\Real_+^d$.
The first of these corresponds to an interesting class of models
which can be thought of as multivariate generalizations of the class of
decreasing and convex
densities on $\Real_+$ treated by \citet{MR1891742}, while the
second, $h(y) = e^{y}$, corresponds
to multivariate versions of the log-convex families studied by
\citet{MR1637480}.
Note that our increasing classes $\Model(y_{+}^{1/r}, \RealP^d )$
with $r>0$ are quite different from the
$r$-concave classes defined above and appear to be completely new,
corresponding instead to
$r$-convex densities on $\RealP^d$.

Here is an outline of the rest of the paper.
All of our main results are presented in Section~\ref{sec:MainReults}.
Section~\ref{sec:DefsAndBasicProperties} gives definitions and basic properties
of the transformations involved.
Section~\ref{sec:ExistenceOfMLEs} % of the Maximum Likelihood
%Estimators}
establishes existence of
the maximum likelihood estimators for both increasing and decreasing
transformations $h$
under suitable conditions on the function $h$.
In Section~\ref{sec:ConsistencyOfMLEs}, % of the Maximum Likelihood
%Estimators}
we give statements concerning consistency of the estimators, both in
the Hellinger metric
and in uniform metrics under natural conditions.
In Section~\ref{sec:LAMLowerBounds}, %ocal Asymptotic Minimax
%LowerBounds}
we present asymptotic minimax
lower bounds for estimation in these classes under natural curvature hypotheses.
We conclude the section with a brief discussion of conjectures
concerning attainability of the
minimax rates by the maximum likelihood estimators.
All of the proofs are given in Section~\ref{sec:Proofs}.

Supplementary material and some proofs omitted here are available
in \citet{SW10supp}.
There, we also summarize a number of
definitions and key results
from convex analysis in an Appendix, Section A.
We use standard notation from convex analysis; see ``Notation'' for a
(partial) list.

%Here is some elaboration on this last point.
%It is fairly well-known that nonparametric maximum likelihood
%estimators of monotone functions
%have desirable adaptation properties with respect to the smoothness of
%the
%underlying true monotone function:
%see \citet{MR1026298} for a study of the Grenander estimator of a
%monotone density;
%see \citet{MR1898714} for the construction of an adaptive
%estimator
% (at a point) in the context of the white noise model;
%and see
%in the setting of convex function estimation,
%but much more remains to be done.
% adapts to the local smoothness
%of $\varphi_0$ in the following sense: if $\varphi_0$ is H\"older$(
%where $1\le\beta\le2$, then $\| \widehat{\varphi}_n - \varphi_0
%for any $A< a<b< B$ and $\| g \|_a^b := \sup_{a \le x \le b} | g(x) |$.
%This carries over similarly to yield local adaptivity properities of $
%Current evidence suggests that the nonparametric
%MLE's of convex densities as studied in \citet{MR1891742} and
%of log-concave densities as in \citet{DuembgenRufibach07} are
%also adaptive to local smoothness
%in terms of their local limiting distributions.
%We intend to investigate this in more detail in future work.
%As was shown by \citet{MR0087249}, unimodality and log--concavity
%share even stronger ties. A density function $f$
%is log--concave if and only if its convolution with any unimodal
%density is again unimodal.
%Therefore, log--concavity
%is often referred to as ``strong unimodality''.

%s1.2 ###
\subsection{Convex-transformed density estimation}
\label{sec:confextransformeddensityestimation}
Now, let $X_1, \ldots, X_n$ be
$n$ independent random variables distributed according to
a probability density $p_0 = h ( g_0 (x)) $ on $\Real^d$, where
$h$ is a fixed monotone (increasing or decreasing) function
and $g_0$ is an (unknown) convex function.
The probability measure on the Borel sets $\mathcal{B}_d$
corresponding to $p_0$ is
denoted by $P_0$.

The maximum likelihood estimator (MLE) of a log-concave density on
$\Real$ was introduced in
\citet{Rufibach06} and
\citet{MR2546798}.
Algorithmic aspects were treated in
\citet{Rufibach07} and, in a more general
framework, in
\citet{DumbHusRuf07}, while
consistency with respect to the Hellinger metric was established
by \citet{PalWoodroofeMeyer07} and rates of convergence of
$\hat{f}_n$ and $\widehat{F}_n$
were established by \citet{MR2546798}.
Asymptotic distribution theory for the MLE of a log-concave density on
$\Real$ was established by
\citet{MR2509075}.
%Since the derivation of the MLE of a log--concave density
%is extensively treated in these references, we only briefly recall its
%definition and the properties relevant
%for this paper.

If $\Conv$ denotes the class of all closed proper convex functions $g\dvtx
\Real^d \to(-\infty,\infty]$,
the estimator $\hat{g}_n$ of
$g_0$ is the maximizer of the functional
\[
\mathbb{L}_n g \equiv\int(\log h)\circ g \,d\mathbb{P}_n
\]
over the class $\ModelC(h)\subset\Conv$ of all convex functions $g$
such that $h\circ g$ is a
density and where $\PP_n$ is the empirical measure of the observations.
The maximum likelihood estimator of the convex-transformed density
$p_0$ is then
%The log--concave
%density estimator is then
$\hat{p}_n := h ( \hat{g}_n )$ when it exists and is unique.
We investigate conditions for existence and uniqueness in
Section~\ref{sec:MainReults}.

%s2 ###
\section{Main results}
\label{sec:MainReults}

%%--------------------------------------------------------------------------------------------------
%s2.1 ###
\subsection{Definitions and basic properties}
\label{sec:DefsAndBasicProperties}

To construct the classes of convex-transformed densities of interest here,
we first need to define two classes of monotone transformations.
An \textit{increasing transformation} $h$ is a nondecreasing function
$\Realc\to\RealPc$ such that $h(-\infty) = 0$
and $h(+\infty)= +\infty$.
%%--------------------------------------------------------------------------------------------------
We define the limit points $y_{0}<y_{\infty}$ of the increasing
transformation $h$ as follows:
$y_{0} = \inf\{y\dvtx h(y)>0\}$,
$y_{\infty} = \sup\{y\dvtx h(y)<+\infty\}$.
% \end{eqnarray*}
We make the following assumptions about the asymptotic behavior of the
increasing transformation:
%%--------------------------------------------------------------------------------------------------
%
\begin{enumerate}[(I.3)]
\item[(I.1)]\hypertarget{aI-integr-zero}
the function $h(y)$ is $o(|y|^{-\alpha})$ for
some $\alpha>d$ as $y\to-\infty$;
\item[(I.2)]
\hypertarget{aI-integr-inf} if $y_{\infty}<+\infty$, then $h(y)\asymp
(y_{\infty}-y)^{-\beta}$ for some $\beta>d$ as $y\uparrow y_{\infty}$;
\item[(I.3)]
\hypertarget{aI-diff}
the function $h$ is continuously differentiable on the
interval $(y_{0}, y_{\infty})$.
\end{enumerate}

%%--------------------------------------------------------------------------------------------------
Note that the assumption \hyperlink{aI-integr-zero}{(I.1)} is satisfied if
$y_{0}>-\infty$.\vadjust{\goodbreak}
\begin{defn}
For an increasing transformation $h$, an
\textit{increasing class of convex-transformed densities}
or simply an increasing model $\Model(h)$ on ${\overline{\mathbb{R}}}{}^{d}_{+}$
is the family of all bounded densities which have the form $h \circ g
\equiv h(g(\cdot))$,
where $g$ is a closed proper convex function with $\dom g = {\overline{\mathbb{R}}}{}^{d}_{+}$.
\end{defn}
%
%%--------------------------------------------------------------------------------------------------
\begin{rem}\label{ch-def-reminc}
Consider a density $h\circ g$ from an increasing model $\Model(h)$.
Since $h\circ g$ is bounded, we have $g< y_{\infty}$.
The function $\tilde g = \max(g, y_{0})$ is convex and $h \circ\tilde
g = h\circ g$.
Thus, we can assume that $g \ge y_{0}$.
\end{rem}
%
%%--------------------------------------------------------------------------------------------------

A \textit{decreasing transformation} $h$ is a nonincreasing function
$\Realc\to\RealPc$ such that $h(-\infty) = +\infty$ and $h(+\infty
) = 0$.
%%--------------------------------------------------------------------------------------------------
We define the limit points $y_{0}>y_{\infty}$ of the decreasing
transformation $h$ as follows:
% \[
$ y_{0} = \sup\{y\dvtx h(y)>0\}$,
$ y_{\infty} = \inf\{y\dvtx h(y)<+\infty\}$.
% \end{eqnarray*}
We make the following assumptions about the asymptotic behavior of the
decreasing transformation:
%%--------------------------------------------------------------------------------------------------
%
\begin{enumerate}[(D.1)]
\item[(D.1)]\hypertarget{aD-integr-zero}
the function $h(y)$ is $o(y^{-\alpha})$ for
some $\alpha>d$ as $y\to+\infty$;
\item[(D.2)]\hypertarget{aD-integr-inffin} if $y_{\infty}>-\infty$, then $h(y)\asymp
(y-y_{\infty})^{-\beta}$ for some $\beta>d$ as $y\downarrow
y_{\infty}$;
\item[(D.3)]\hypertarget{aD-integr-infinf} if $y_{\infty}=-\infty$, then $h(y)^{\gamma
}h(-Cy)=o(1)$ for some $\gamma, C>0$ as $y\to-\infty$;
\item[(D.4)]\hypertarget{aD-diff} the function $h$ is continuously differentiable on the
interval $(y_{\infty},y_{0})$.
\end{enumerate}

%%--------------------------------------------------------------------------------------------------
Note that the assumption \hyperlink{aD-integr-zero}{(D.1)} is satisfied if
$y_{0}<+\infty$.
We now define the decreasing class of densities $\Model(h)$.
%However, in this case we do not need to restrict the domain of convex
%functions.
%%--------------------------------------------------------------------------------------------------
%
\begin{defn}
For a decreasing transformation $h$, a \textit{decreasing class of
convex-transformed densities}
or simply a decreasing model $\Model(h)$ on $\Real^{d}$ is the family
of all bounded densities
which have the form $h\circ g$, where $g$ is a closed proper convex
function with $\dim(\dom g) = d$.
\end{defn}
%
%%--------------------------------------------------------------------------------------------------
\begin{rem}
Consider a density $h \circ g$ from a decreasing model $\Model(h)$.
Since $h\circ g$ is bounded, we have $g> y_{\infty}$.
For the sublevel set $C = \lev_{y_{0}} g$, the function
$\tilde g = g+\delta(\cdot|C)$ is convex and $h\circ\tilde
g = h \circ g$.
Thus, we can assume that $\lev_{y_{0}} g = \dom g$.
\end{rem}

For a monotone transformation $h$, we denote by $\ModelC(h)$ the class
of all
closed proper convex functions $g$
such that $h \circ g$ belongs to a monotone class $\Model(h)$.
%There is an analogous definition of an \textit{increasing transformation}
%$h$, and an increasing
%class $\Model(h)$ which has been explored in ??. In these models the
%domain of the convex functions $g$
%must be restricted to proper subsets of $\Real^d$ such as $\Real_+^d$.
The following lemma allows us to compare models defined by increasing
or decreasing transformations $h$.
%%--------------------------------------------------------------------------------------------------
%
\begin{lem}\label{ch-def-comp}
Consider two decreasing (or increasing) models $\Model(h_1)$ and
$\Model(h_{2})$.
If $h_{1}=h_{2}\circ f$ for some convex function $f$, then $\Model
(h_{1}) \subseteq\Model(h_{2})$.
\end{lem}
%
%%--------------------------------------------------------------------------------------------------
\begin{pf}
The argument below is for a decreasing model. For an increasing model,
the proof is similar. If $f(x)>f(y)$ for some $x<y$, then $f$ is
decreasing on $(-\infty,x)$, $f(-\infty)=+\infty$ and
therefore $h_{2}$ is constant on $(f(x),+\infty)$, and we can redefine
$f(y)=f(x)$ for all $y<x$.
Thus, we can always assume that $f$ is nondecreasing.\vadjust{\goodbreak}

For any convex function $g$, the function $f\circ g$ is
also convex. Therefore, if $p = h_1\circ g\in\Model(h_{1})$, %
then $p = h_2\circ f\circ g \in\Model(h_{2})$.
\end{pf}

In this section, we discuss several examples of monotone models. The
first two families
are based on increasing transformations $h$.
%%--------------------------------------------------------------------------------------------------
%
\begin{ex}[(Log-convex densities)]
This increasing model is defined by $h(y) = e^y$.
Limit points are $y_0 = -\infty$
and $y_{\infty} = \infty$. Assumption \hyperlink{aI-integr-zero}{(I.1)} holds for any $\alpha
>d$. These classes of densities
were considered by \citet{MR1637480}, who established several
useful preservation properties.
In particular, log-convexity is preserved under mixtures [\citet
{MR1637480}, Proposition 3]
and under marginalization [\citet{MR1637480}, Remark 8, page 361].
\end{ex}
%
%%--------------------------------------------------------------------------------------------------
\begin{ex}[($r$-convex densities)]
This family of increasing models is defined by the transforms
$h(y)=\max(y,0)^{s} = y_+^{s}$ with $s>0$.
Limit points are $y_{0}=0$ and $y_{\infty}=+\infty$.
Assumption \hyperlink{aI-integr-zero}{(I.1)} holds for any $\alpha>d$.
As noted in Section~\ref{sec:intro},
%%sec:logconcaveAndrconcavedensities}
the model $\Model(y_{+}^{1/r}, \RealP^d)$
corresponds to the class of $r$-convex densities, with $r=\infty$
corresponding to the log-convex densities of the previous example.
For $r < \infty$, these classes appear not to have been previously
discussed or considered, except in special cases: the case $r=1$ and
$d=1$ corresponds to the
class of decreasing convex densities on $\RealP$ considered by
\citet{MR1891742}.
It follows from Lemma~\ref{ch-def-comp} that
%
%e2.1 ###
\begin{equation}\label{ch-exm-incl-rConvex}
\Model(e^{y}, \RealP^d ) \subset\Model(y_{+}^{s_{2}}, \RealP^d
)\subset\Model(y_{+}^{s_{1}}, \RealP^d)\qquad
\mbox{for }  0 < s_1 < s_2 < \infty.
\end{equation}
\end{ex}

%%--------------------------------------------------------------------------------------------------
We now consider some models based on decreasing transformations $h$.
%%--------------------------------------------------------------------------------------------------
%
\begin{ex}[(Log-concave densities)]
\label{ch-def-exm-lc}
This decreasing model is defined by the transform $h(y) = e^{-y}$.
Limit points are $y_{0}=+\infty$ and $y_{\infty}=-\infty$.
Assumption~\hyperlink{aD-integr-zero}{(D.1)} holds for any $\alpha>d$.
Assumption \hyperlink{aD-integr-infinf}{(D.3)} holds for any $\gamma>C>0$.
\end{ex}

Many parametric models are subsets of this model: in particular,
uniform, Gaussian,
gamma, beta, Gumbel, Fr\'echet and logistic densities are all log-concave.
%Below we specify convex functions $g$ which correspond to the
%densities of several distributions:
%$C$ is log-concave:
%log-concave for $r>1$:
%log-concave for $\alpha,\beta>1$:
%Gumbel, Fr\'echet and logistic distributions also have log-concave
%densities.
%%--------------------------------------------------------------------------------------------------
%
\begin{ex}[($r$-concave densities and power-convex densities)]
\label{ch-def-exm-pc}
This family of decreasing models is defined by the transforms
$h(y)=y_+^{-s}$ for $s>d$.
Limit points are $y_{0}=+\infty$ and $y_{\infty}=0$.
Assumption \hyperlink{aD-integr-zero}{(D.1)} holds for any $\alpha\in(d,s)$.
Assumption \hyperlink{aD-integr-inffin}{(D.2)} holds for $\beta=s$.
As noted in Section~\ref{sec:intro},
%%sec:logconcaveAndrconcavedensities}
the model $\Model(y_{+}^{1/r}) = \Model(y_{+}^{-s})$ (with $r =-1/s < 0$)
corresponds to the class of $r$-concave densities.
% with $r=\infty$
%corresponding to the log-convex densities of the previous example.
From Lemma~\ref{ch-def-comp}, we have the following inclusion:
%
%e2.2 ###
\begin{equation}\label{ch-exm-incl}
\Model(e^{-y}) \subset\Model(y_{+}^{-s_{2}})\subset\Model
(y_{+}^{-s_{1}}) \qquad\mbox{for }   s_1 < s_2.
\end{equation}
\end{ex}

The models defined by power transformations include some parametric
models with heavier-than-exponential tails.
Several examples, including the multivariate generalizations of Pareto,
Student-$t$, and
$F$-distributions are discussed in \citet{MR0404559}---none of
these families are log-concave; see \citet{JoKo} and \citet
{SW10supp} for explicit computations.
%Density of a multivariate Pareto distribution $(\theta,a)$ is
%power-convex for $s\in(d,a+d]$:
%)^{1/s}(1 - d + \sum x_{i}/\theta_{i})^{(a+d)/s}. \]
%Density of a multivariate $t$-distribution $(d, n, \mu, \Sigma)$ is
%power-convex for $s\in(d,n+d]$:
%(1 + \frac{1}{n}(x-\mu)^{T}\Sigma^{-1}(x-\mu))^{(d+n)/2s}.
%Density of a multivariate $F$-distribution $(n_{0}, n_{1},
%(d,(n_{0}/2)+d]$:
%)^{n/2s} (\prod_{i=1}^{d} x_{i}^{2-n_{i}})^{1/2s},\]
%where $n=\sum_{i=0}^{d} n_{i}$.
%Since the distributions above belong to the power-convex models
%only for bounded values of the parameter $s$ the inclusion
%they do not belong to the log-concave model (corresponding to $s=+

\citet{MR0404559} developed
a framework which unifies log-concave and power-convex
densities and gives an interesting characterization for these classes.
Here, we briefly state the main result.
\begin{defn}
Let $C \subseteq\Real^{d}$ be an open convex set and let $s\in\Realc$.
We then define $\mathcal{M}_{s}(C)$ as the family of all positive Radon
measures $\mu$ on $C$ such that
%
%e2.3 ###
\begin{equation}
\label{borell-cond}
\mu_{*}\bigl(\theta A + (1-\theta)B\bigr) \ge[\theta\mu_{*}(A)^{s} +
(1-\theta)\mu_{*}(B)^{s}]^{1/s}
\end{equation}
holds for all $\varnothing\neq A,B\subseteq C$ and all $\theta\in(0,1)$.
We define $\mathcal{M}^{\circ}_{s}(C)$ as a subfamily of $
\mathcal{M}_{s}(C)$ which consists of probability measures such that
the affine hull of its support has dimension $d$.
Here, $\mu_{*}$ is the inner measure corresponding to $\mu$ and the
cases $s=0,\infty$ are defined by continuity.
\end{defn}

One of the main results of \citet{MR0404559}, \citet
{MR0404557} and \citet{MR0428540} is as follows.
\begin{theorem}[(Borell, Prekopa, Rinott)]
\label{thm:BorellsTheorem}
For $s<0$, the family $\mathcal{M}^{\circ}_{s}(\Real^{d})$ coincides with
the power-convex family $\Model(y_{+}^{-d+1/s})$.
For $s=0$, the family $\mathcal{M}^{\circ}_{0}(\Real^{d})$
coincides with the log-concave family $\Model(e^{-y})$.
This continues to hold if (\ref{borell-cond}) holds with $\theta=
1/2$ for all
compact (or open, or semi-open) blocks $A,B \subseteq\Omega$ (i.e., rectangles
with sides parallel to the coordinate axes).
\end{theorem}

%Finally, Corollary 3.1 \citet{MR0404559} says that the condition
%Let $\Omega\subseteq\Real^{d}$ be an open convex set and let
%$\mu$ be a positive Radon measure in $\Omega$.
%Then $\mu\in\mathcal{M}_{s}(\Omega)$ if and only if
%holds for all compact (or open, or semiopen) blocks $A_{1},A_{2}
%(i.e. rectangles with sides parallel to the coordinate axes).

Theorem~\ref{thm:BorellsTheorem} provides a special case of what has
come to be known as the
\textit{Borell--Brascamp--Lieb} inequality; see, for example,
%http://en.wikipedia.org/wiki/BorellBrascampLieb_inequality,
\citet{MR954608}
and \citet{MR0450480}.
The current terminology is apparently due to \citet{MR1865396}.

%s2.2 ###
\subsection{Existence of the maximum likelihood estimators}
\label{sec:ExistenceOfMLEs}

Now, suppose that $X_1, \ldots, X_n $ are i.i.d. with density $p_0 (x)
= h (g_0(x))$
for a fixed monotone transformation $h$ and a convex function $g_0$.
As before, $\PP_n = n^{-1} \sum_{i=1}^n \delta_{X_i}$
is the empirical measure of
the $X_i$'s and $P_0$ is the probability measure corresponding to $p_0$.
Then, $\mathbb{L}_n g = \PP_n \log h\circ g $ is the log-likelihood
function (divided by $n$) and
\[
\hat{p}_n \equiv\argmax\{ \mathbb{L}_n g \dvtx h\circ g \in
\Model(h) \}
\]
is the \textit{maximum likelihood estimator of $p$ over the class $\Model
(h)$}, assuming it exists
and is unique. We also write $\hat{g}_n$ for the MLE of $g$.
We first state our main results concerning existence and uniqueness of
the MLEs for the classes $\Model(h)$.
%
%%--------------------------------------------------------------------------------------------------
%
\begin{theorem}
\label{mle-iexi-main} Suppose that $h$ is an increasing transformation
satisfying assumptions \textup{\hyperlink{aI-integr-zero}{(I.1)}--\hyperlink{aI-diff}{(I.3)}}. %for
%an increasing transformation $h$
The MLE $\hat{p}_n$ then exists almost surely for the model
$\Model(h)$.
\end{theorem}
%
%%--------------------------------------------------------------------------------------------------
%
\begin{theorem}
\label{mle-dexi-main}
Suppose that $h$ is a decreasing transformation satisfying assumptions
\textup{\hyperlink{aD-integr-zero}{(D.1)}--\hyperlink{aD-diff}{(D.4)}}.
The MLE $\hat{p}_n$ then exists almost surely for the model
$\Model(h)$ if
\[
n \ge n_d \equiv d + d\gamma1_{\{-\infty\}} (y_{\infty})
+ \frac{\beta d^2}{\alpha(\beta-d)} 1\{ y_{\infty} > -\infty\} .
\]
\end{theorem}

%%--------------------------------------------------------------------------------------------------
Uniqueness of the MLE is known for the log-concave model $\Model
(e^{-y})$; see, for example,
\citet{MR2546798} for $d=1$ and \citet
{culeSamworthStewart2010} for $d\ge1$.
For a brief further comment, see Section~\ref{sec:ConjUnique}.
%%--------------------------------------------------------------------------------------------------

%s2.3 ###
\subsection{Consistency of the maximum likelihood estimators}
\label{sec:ConsistencyOfMLEs}

Once existence of the MLEs is ensured, our attention shifts to
other properties of the estimators: our main concern in this subsection is
consistency. While, for a decreasing model, it is possible to prove consistency
without any restrictions, for an increasing model, we need the
following assumptions
about the true density $h\circ g_{0}$:
\begin{enumerate}[(I.3)]
\item[(I.4)]\hypertarget{aI-dens-bdd}
the function $g_{0}$ is bounded by some constant
$C<y_{\infty}$;
\item[(I.5)]\hypertarget{aI-dens-logsup}
if $d>1$, then we have, with $\pnorm{x} \equiv
\prod_{j=1}^d x_j$ for $x \in\RealP^d$,
\[
C_{g}\equiv\int_{\RealP^{d}}\log\biggl(\frac{1}{\pnorm{x} \wedge
1}\biggr)\,dP_0 (x) <\infty.
\]
\end{enumerate}
\begin{rem}
Note that for $d=1$, the assumption \hyperlink{aI-dens-logsup}{(I.5)} follows
from
assumption \hyperlink{aI-dens-bdd}{(I.4)} and integrability of $\log(1/x)$ at zero.
This assumption is also true if $P$ has finite marginal densities.
\end{rem}
\begin{enumerate}[(I.6)]
\item[(I.6)]\hypertarget{aI-dens-loginf} We have
$\int_{\RealP^{d}}(h |{\log h} |)\circ g_{0}(x)\,dx <\infty$.
\end{enumerate}

Let $H(p,q)$ denote the Hellinger distance between two probability
measures with densities
$p$ and $q$ with respect to Lebesgue measure on $\Real^{d}$:
%
%e2.4 ###
\begin{equation}\quad
H^{2}(p,q)\equiv\frac{1}{2}\int_{\Real^{d}}\bigl(\sqrt{p}(x) - \sqrt
{q}(x)\bigr)^{2}\,dx = 1 - \int_{\Real^{d}} \sqrt{p(x)q(x)}\,dx.
\end{equation}
Our main results about increasing models are as follows.
%
%%--------------------------------------------------------------------------------------------------
%
\begin{theorem}[(S.2.2)]%\ref{supp:thm:IncHellingerConsistencyOfMLE})]
\label{thm:IncHellingerConsistencyOfMLE}
For an increasing model $\Model(h)$, where
$h$ satisfies assumptions
\textup{\hyperlink{aI-integr-zero}{(I.1)}--\hyperlink{aI-diff}{(I.3)}} and for the true density
$h\circ g_{0}$ which satisfies assumptions \textup{\hyperlink{aI-dens-bdd}{(I.4)}--\hyperlink{aI-dens-loginf}{(I.6)}},
the sequence of MLEs $\{\hat{p}_n = h\circ\hat{g}_{n}\}$ is
Hellinger consistent:
$H( \hat{p}_n, p_0 ) = H(h\circ\hat g_{n},h\circ g_{0})\to
_{a.s.} 0$.
\end{theorem}
%
%%--------------------------------------------------------------------------------------------------
%%--------------------------------------------------------------------------------------------------
%
\begin{theorem}
\label{thm:IncPointwiseAndUniformConsistency}
For an increasing model $\Model(h)$, where
$h$ satisfies assumptions
\textup{\hyperlink{aI-integr-zero}{(I.1)}--\hyperlink{aI-diff}{(I.3)}}, and for the true density
$h\circ g_{0}$ which satisfies assumptions \textup{\hyperlink{aI-dens-bdd}{(I.4)}--\hyperlink{aI-dens-loginf}{(I.6)}},
the sequence of MLEs $\hat g_{n}$ is
pointwise consistent. That is, $\hat g_{n}(x)\to_{a.s.} g_{0}(x)$ for
$x\in\ri(\RealP^{d})$ and convergence is uniform on compacta.
\end{theorem}
%
%%--------------------------------------------------------------------------------------------------

The results about decreasing models can be formulated in a similar way.
%%--------------------------------------------------------------------------------------------------
%
\begin{theorem}
\label{thm:DecHellingerConsistencyOfMLE}
For a decreasing model $\Model(h)$, where
$h$ satisfies assumptions
\textup{\hyperlink{aD-integr-zero}{(D.1)}--\hyperlink{aD-diff}{(D.4)}},
the sequence of MLEs $\{\hat{p}_n = h\circ\hat{g}_{n}\}$ is
Hellinger consistent:
\[
H( \hat{p}_n, p_0 ) = H(h\circ\hat g_{n},h\circ g_{0})\to
_{a.s.} 0.
\]
\end{theorem}
%
%%--------------------------------------------------------------------------------------------------
%%--------------------------------------------------------------------------------------------------
%
\begin{theorem}
\label{thm:DecPointwiseAndUniformConsistency}
For a decreasing model $\Model(h)$ with
$h$ satisfying assumptions
\textup{\hyperlink{aD-integr-zero}{(D.1)}--\hyperlink{aD-diff}{(D.4)}},
the sequence of MLEs $\hat g_{n}$ is
pointwise consistent in the following sense.
Define $g_{0}^{*} = g_{0} + \delta(\cdot| \ri(\dom g_{0}))$.
Then, $g_{0}^{*}=g_{0}$ a.e.,
$\hat g_{n}\to_{a.s.} g_{0}^{*}$
and the convergence is uniform on compacta.
Moreover, if $\dom g_{0}=\Real^{d}$, then
$\|h \circ\hat g_{n}- h\circ g_{0}\|_{\infty}\to_{a.s.} 0 $.
\end{theorem}
%
%%--------------------------------------------------------------------------------------------------

%s2.4 ###
\subsection{Local asymptotic minimax lower bounds}
\label{sec:LAMLowerBounds}

In this section, we establish local asymptotic minimax lower bounds
for any estimator of several functionals of interest on the family
$\Model(h)$
of convex-transformed densities.
%we use \textit{local minimax risk} to bound from below the rate of
%convergence of any estimator which estimates a given functional on the
%family
%$\Model(h)$ of convex-transformed densities.
We start with several general results following \citet{MR1792307}
and then apply them to estimation at a fixed point and to mode estimation.

First, we define minimax risk as in \citet{MR1105839}.
\begin{defn}
Let $\mathcal{P}$ be a class of densities on $\Real^{d}$ with respect
to Lebesgue
measure and let $T$ be a functional $T\dvtx\mathcal{P}\to\Real$.
For an increasing convex loss function $l$ on $\Real_+$, we define the
\textit{minimax risk} as
%
%e2.5 ###
\begin{equation}
R_{l}(n;T,\mathcal{P}) = \inf_{t_{n}}\sup_{p\in\mathcal{P}} \E
_{p^{\times n}} l\bigl(|t_{n}(X_{1},\ldots,X_{n}) - Tp|\bigr),
\end{equation}
where $t_{n}$ ranges over all possible estimators of $Tp$ based on
$X_{1},\ldots,X_{n}$.
\end{defn}

The main result (Theorem 1) in \citet{MR1792307} can be
formulated as follows.
%%--------------------------------------------------------------------------------------------------------------------
%
\begin{theorem}[(Jongbloed)]\label{lb-gen}
Let $\{p_{n}\}$ be a sequence of densities in $\mathcal{P}$ such that
$\limsup_{n\to\infty}\sqrt{n} H(p_{n},p)\le\tau$
for some density $p$ in $\mathcal{P}$. Then,
%
%e2.6 ###
\begin{equation}
\liminf_{n\to\infty} \frac{R_{l}(n;T,\{p,p_{n}\})}{l((1/4)e^{-2\tau
^{2}}|T(p_{n}) - T(p)|)}\ge1.\vadjust{\goodbreak}
\end{equation}
\end{theorem}

%%--------------------------------------------------------------------------------------------------------------------
It will be convenient to reformulate this result in the following form.
%%--------------------------------------------------------------------------------------------------------------------
%
\begin{cor}\label{lb-cor}
Suppose that for any $\varepsilon>0$ small enough, there exists
$p_{\varepsilon}\in\mathcal{P}$ such that for some $r>0$,
$\lim_{\varepsilon\to0}\varepsilon^{-1}|Tp_{\varepsilon}-Tp| = 1$ %
and
\[
\limsup_{\varepsilon\to0}\varepsilon^{-r}H(p_{\varepsilon},p)\le c.
\]
There then exists a sequence $\{p_{n}\}$ such that
%
%e2.7 ###
\begin{equation}
\liminf_{n\to\infty} n^{1/2r}R_{1}(n;T,\{p,p_{n}\})\ge\frac
{1}{4(2re)^{1/2r}}c^{-1/r},
\end{equation}
where $R_{1}$ is the risk which corresponds to $l(x) = |x|$.
\end{cor}

Corollary~\ref{lb-cor} shows that for a fixed change in the value of
the functional
$T$, a~family $p_{\varepsilon}$ which is closer to the true density $p$
with respect to Hellinger distance provides a sharper lower bound.
This suggests that for the functional $T$ which depends only on the
local structure
of the density, we would like our family $\{p_{\varepsilon}\}$ to
deviate from $p$ also locally.
Below, we formally define such local deviations.
%%--------------------------------------------------------------------------------------------------------------------
%
\begin{defn}
We call a family of measurable functions $\{p_{\varepsilon}\}$ a \textit{deformation}
of a measurable function $p$ if $p_{\varepsilon}$ is defined for any
$\varepsilon>0$ small enough,
$\lim_{\varepsilon\to0}\esssup|p-p_{\varepsilon}|=0$
and there exists a bounded family of real numbers $r_{\varepsilon}$
and a point $x_{0}$ such that
\[
\mu[{\supp}|p_{\varepsilon}(x)-p(x)|] > 0,\qquad
{\supp}|p_{\varepsilon}(x)-p(x)| \subseteq B(x_{0},r_{\varepsilon}).
\]
If, in addition, we have
$\lim_{\varepsilon\to0}r_{\varepsilon} = 0$,
then %\end{eqnarray*}
we say that $\{p_{\varepsilon}\}$ is a \textit{local deformation} at $x_{0}$.
\end{defn}

Since, for a deformation $p_{\varepsilon}$, we have
$\mu[{\supp}|p_{\varepsilon}(x)-p(x)|] > 0$ for every \mbox{$\varepsilon>0$},
there exists $\delta>0$ such that $\mu\{x\dvtx|p_{\varepsilon
}(x)-p(x)|>\delta\}>0$
and thus the $L_{r}$-distance from $p_{\varepsilon}$ to $p$ is
positive for all $\varepsilon>0$.
Note that this is always true if $p$ and $p_{\varepsilon}$ are
continuous at
$x_{0}$ and $p_{\varepsilon}(x_{0})\neq p(x_{0})$.

We can now state our lower bound for estimation of the convex-transformed
density value at a fixed point $x_{0}$.
This result relies on the properties of strongly convex functions, as
described in
Appendix S.A.4,
and can be applied to both
increasing and
decreasing classes of convex-transformed densities.
%%--------------------------------------------------------------------------------------------------------------------
%
\begin{theorem}\label{lb-point-main}
Let $h$ be a monotone transformation, let $p = h\circ g\in\Model(h)$
be a convex-transformed density and suppose that $x_0$ is
a point in $ \ri(\dom g)$
such that $h$ is continuously differentiable at $g(x_{0})$, $h\circ g(x_{0})>0$,
$h'\circ g(x_{0})\neq0$ and $\curv_{x_{0}} g > 0$.
Then, for the functional\vadjust{\goodbreak} $T(h\circ g)\equiv g(x_{0})$, there exists a sequence
$\{p_{n}\}\subset\Model(h)$ such that
%
%e2.8 ###
\begin{equation}
\liminf_{n\to\infty} n^{{2}/({d+4})}
R_{1}(n;T,\{h\circ g, p_{n}\})
\ge C(d)\biggl[\frac{h\circ g(x_{0})^{2}\curv_{x_{0}} g}{h'\circ
g(x_{0})^{4}}\biggr]^{{1}/({d+4})},\hspace*{-32pt}
\end{equation}
where the constant $C(d)$ depends only on the dimension $d$.
\end{theorem}
%
%%--------------------------------------------------------------------------------------------------------------------
\begin{rem}
If, in addition, $g$ is twice continuously differentiable at $x_{0}$ and
$\nabla^{2}g(x_{0})$ is positive definite, then, by Lemma S.A.22,
we have $\curv_{x_{0}} g = \det(\nabla^{2} g(x_{0}))$.
\end{rem}
%
%---------------------------------------------------

In \citet{MR1792307}, lower bounds were constructed for functionals
with values in $\Real$.
However, it is easy to see that the proof does not change for functionals
with values in an arbitrary metric space $(V,s)$ if, instead of $|Tp -
Tp_{n}|$, we consider $s(Tp, Tp_{n})$.
We define
%
%e2.9 ###
\begin{equation}
R_{s}(n;T,\mathcal{P}) = \inf_{t_{n}}\sup_{p\in\mathcal{P}} \E
_{p^{\times n}} s(t_{n}(X_{1},\ldots,X_{n}), Tp)
\end{equation}
and the analog of Corollary~\ref{lb-cor} then has the following form.
\begin{cor}\label{lb-cor2}
Suppose that for any $\varepsilon>0$ small enough,
there exists $p_{\varepsilon}\in\mathcal{P}$ such that for some $r>0$,
\[
\lim_{\varepsilon\to0}\varepsilon^{-1}s(Tp_{\varepsilon},Tp) = 1,\qquad
\limsup_{\varepsilon\to0}\varepsilon^{-r}H(p_{\varepsilon},p)\le c.
\]
There then exists a sequence $\{p_{n}\}$ such that
%
%e2.10 ###
\begin{equation}
\liminf_{n\to\infty} n^{1/2r}R_{s}(n;T,\{p,p_{n}\})\ge\frac
{1}{4(2re)^{1/2r}}c^{-1/r}.
\end{equation}
\end{cor}

We now consider estimation of the functional
$T(h \circ g) = \argmin(g) \in\Real^{d}$
for the density
$p = h\circ g\in\Model(h)$, assuming that the minimum is unique.
This is equivalent to estimation of the mode of $p = h \circ g$.

The construction of a lower bound for the functional $T$ is
similar to the procedure we presented for estimation of $p = h\circ g$
at a fixed point $x_0$.
Again, we use two opposite deformations: one is local and changes the
functional value,
the other is a convex combination with a fixed deformation and
negligible-in-Hellinger-distance computation.
However, in this case, the
minimax rate also depends on the growth rate of $g$.
%%--------------------------------------------------------------------------------------------------------------------
%
\begin{theorem}[(S.3.4)]%[S.\ref{supp:thm:lb-mode-main}]
\label{thm:lb-mode-main}
Let $h$ be a decreasing transformation, $h\circ g\in\Model(h)$ be a
convex-transformed density and a point $x_{0}\in\ri(\dom g)$ be a
unique global minimum of $g$ such that $h$ is continuously
differentiable at
$g(x_{0})$, $h'\circ g(x_{0})\neq0$
and $\curv_{x_{0}} g > 0$.
In addition, let us assume that $g$ is locally H\"older continuous
at~$x_{0}$, that is,
$|g(x)-g(x_{0})|\le L\|x-x_{0}\|^{\gamma}$
with respect to some norm \mbox{$\|\cdot\|$}.
Then, for the functional $T(h \circ g)\equiv\argmin g$, there\vadjust{\goodbreak}
exists a sequence $\{p_{n}\}\in\Model(h)$ such that
%
%e2.11 ###
\begin{eqnarray}
&&\liminf_{n\to\infty}
n^{{2}/({\gamma(d+4)})}R_{s}(n;T,\{p,p_{n}\})\nonumber\\[-8pt]\\[-8pt]
&&\qquad
\ge C(d)L^{-{1}/{\gamma}}\biggl[\frac{h\circ g(x_{0})^{2}\curv
_{x_{0}} g}{h'\circ g(x_{0})^{4}}\biggr]^{{1}/({\gamma(d+4)})},\nonumber
\end{eqnarray}
where the constant $C(d)$ depends only on the dimension $d$, and the
metric $s(x,y)$ is defined as $\|x-y\|$.
\end{theorem}
%
%%--------------------------------------------------------------------------------------------------------------------
\begin{rem}
If, in addition, $g$ is twice continuously differentiable at $x_{0}$
and $\nabla^{2}g(x_{0})$
is positive definite, then, by Lemma S.A.22, %S.\ref{supp:cvx-curv-prop},
we have $\curv_{x_{0}} g = \det(\nabla^{2} g(x_{0}))$
and $g$ is locally H\"older continuous at $x_{0}$ with exponent
$\gamma=2$ and any constant $L>\|\nabla^{2} g(x_{0})\|$.
\end{rem}
%
%%--------------------------------------------------------------------------------------------------------------------
\begin{rem}
Since $\curv_{x_{0}} g>0$, there exists a constant $C$ such that
$C\|x-x_{0}\|^{2}\le|g(x)-g(x_{0})|$
and thus we have $\gamma\in(0,2]$.
\end{rem}

%s2.5 ###
\subsection{Conjectures concerning uniqueness of MLEs}
\label{sec:ConjUnique}
There exist counterexamples to uniqueness for nonconvex transformations
$h$ which satisfy assumptions \hyperlink{aD-integr-zero}{(D.1)}--\hyperlink{aD-diff}{(D.4)}.
They suggest that uniqueness of the MLE does not depend on
the tail behavior of the transformation $h$, but rather on the local
properties of $h$ in neighborhoods of the optimal values $\hat
g_{n}(X_{i})$. We conjecture that uniqueness holds for all monotone
models if $h$ is convex and $h/|h'|$ is nondecreasing convex. Further
work on these uniqueness issues is needed.

%s2.6 ###
\subsection{Conjectures about rates of convergence for the MLEs}

We conjecture that the (optimal) rate of convergence $n^{2/(d+4)}$
appearing in
Theorem~\ref{lb-point-main} for estimation of $f(x_0)$ will be
achieved by the MLE only for
$d=2,3$. For $d=4$, we conjecture that the MLE will come within a
factor $(\log n)^{-\gamma}$ (for some $\gamma>0$)
of achieving the rate $n^{1/4}$, but for $d>4$, we conjecture that the
rate of convergence will
be the suboptimal rate $n^{1/d}$. This conjectured rate-suboptimality
raises several interesting further
issues:
\begin{itemize}
\item Can we find alternative estimators (perhaps via penalization or
sieve methods) which achieve the
optimal rates of convergence?
\item For interesting subclasses, do maximum likelihood estimators
remain rate-optimal?
\end{itemize}
%
%generalizations of our classes $\Model(h)$ in
%which both $h$ and $g$ are unknown with $h : \Real^m \rightarrow
%While they achieve optimal rates of estimation in their more general
%problem with a class of estimators
%based on multiple testing and model selection, their procedures will
%very likely have several drawbacks
%relative to the methods we have studied here, including (a) not
%belonging to the the classes of densities
%defined by the models, and (b) difficulties in computation or
%implementation.

%s3 ###
\section{Proofs}
\label{sec:Proofs}

%%%%%%%%%%%%%%%%%%%%%%%%%%%%%%%%%%%%%%%%%%%%%%%%%%%%%%%%%%%%%%%%%%%%%%%%%%%%%%%%%%%%%%%%%%%%%%%%%%%%
%s3.1 ###
\subsection{Preliminaries: Properties of decreasing transformations}

%%--------------------------------------------------------------------------------------------------
%
\begin{lem}
\label{ch-def-dlev}
Let $h$ be a decreasing transformation and $g$ be a closed proper
convex function such that
$\int_{\Real^{d}} h\circ g \,dx = C < \infty$.
The following are then true:\vadjust{\goodbreak}
\begin{enumerate}
\item%\label{ch-def-dlev1}
for $y<+\infty$, the sublevel sets $\lev_{y} g$ are bounded and we have
\[
\mu[\lev_{y} g] \le C/h(y);
\]
\item%\label{ch-def-dlev2}
the infimum of $g$ is attained at some
point $x\in\Real^{d}$.
\end{enumerate}
\end{lem}
%
%%--------------------------------------------------------------------------------------------------
\begin{pf}
1. %\item[\ref{ch-def-dlev1}.]
We have
\begin{eqnarray*}
C &=& \int_{\Realc^{d}} h \circ g \,dx \ge\int_{\lev_{y} g} h \circ g
\,dx \ge h(y) \mu[\lev_{y} g],\\
\mu[\lev_{y} g] &\le& C/h(y).
\end{eqnarray*}
The sublevel set $\lev_{y} g$ has the same dimension as $\dom g$
[Theorem 7.6 in \citet{MR0274683}], which is $d$. By Lemma S.A.1,
%S.\ref{supp:cvx-gen-meas},
this set is bounded when $y<y_{0}$. Therefore, it is enough to prove that
$\lev_{y_{0}} g$ is bounded for $y_{0}<+\infty$.

Since $h\circ g$ is a density, we have $\inf g < y_{0}$.
If $g$ is constant on $\dom g$, then, for all $y\in[\inf g, +\infty)$,
we have $\lev_{y} g = \lev_{\inf g} h$ and it is therefore bounded.
Otherwise, we can choose $\inf h\le y_{1}<y_{2}<y_{0}$.
Then, $\mu[\lev_{y_{2}} g]<\infty$ and, by Lemma S.A.3,
%S.\ref{supp:cvx-gen-levbdd},
we have $\mu[\lev_{y_{0}} g]<\infty$.
The argument above shows that $\lev_{y_{0}} g$ is also bounded.

2. %\item[\ref{ch-def-dlev2}.]
This follows from the fact that $g$ is continuous and $\lev_{y} g$ is bounded
and nonempty for $y>\inf g$.
\end{pf}
%
%%--------------------------------------------------------------------------------------------------
%
\begin{lem}
\label{ch-def-dint}
Let $h$ be a decreasing transformation, let $g$ be a closed proper
convex function
on $\Real^{d}$ and let $Q$ be a $\sigma$-finite Borel measure on
$\Real^{d}$.
Then
\[
\int_{(\lev_{a} g)^{c}} h \circ g \,dQ
= - \int_{a}^{+\infty} h'(y)Q[\lev_{y} g \cap(\lev_{a} g)^{c}]\,dy.
\]
\end{lem}
%
%%--------------------------------------------------------------------------------------------------
%
\begin{pf}
Using the Fubini--Tonelli theorem, we have, with $L_a^c \equiv(\lev_a g)^c$,
\begin{eqnarray*}
\int_{L_a^{c}} h \circ g \,dQ
&=& \int_{L_a^{c}} \int_{0}^{h(a)} 1\{z\le h \circ g(x)\} \,dz\,dQ(x) \\
&=& \int_{L_a^{c}}   \int_{0}^{h(a)} 1\{h^{-1}(z)\ge g(x)\}
\,dz\,dQ(x)\\
&=&-\int_{L_a^{c}} \int_{a}^{\infty} h'(y)1\{y\ge g(x)\} \,dy\,dQ(x)\\
&=&-\int^{\infty}_{a} h'(y)\int_{L_a^c}1\{y\ge g(x)\} \,dQ(x)\,dy\\
&=& -\int^{\infty}_{a} h'(y)Q[\lev_{y} g\cap L_a^{c}]\,dy.
\end{eqnarray*}
\upqed\end{pf}
%
%%--------------------------------------------------------------------------------------------------
%
\begin{lem}
\label{ch-def-dinfval}
Let $h$ be a decreasing transformation and let $g$ be a closed proper
convex function such that
$\int_{\Real^{d}} h \circ g \,dx < \infty$.
Then, $\inf g>y_{\infty}$.\vadjust{\goodbreak}
\end{lem}
%
%%--------------------------------------------------------------------------------------------------
\begin{pf}
Since $g$ is proper, the statement is trivial for $y_{\infty}=-\infty$,
so we assume that $y_{\infty}>-\infty$. If, for $x_{0}$, we have
$g(x_{0})=y_{\infty}$, then there exists a ball $B\equiv B(x;r)$ such that
$g< y_{\infty}+\varepsilon$ on $B$.
Consider the convex function $f$ defined as
$f(x) = y_{\infty} + ({\varepsilon}/r)\|x-x_{0}\| + \delta(x\mathop|B)$.
Then, by convexity, $f \ge g$ and
$\int_{\Real^{d}} h \circ g \,dx \ge\int_{\Real^{d}} h \circ f \,dx$.
We have $\mu[\lev_{y} f] = S(y-y_{\infty})^{d}$ for
$y\in[y_{\infty},y_{\infty}+\varepsilon]$, where $S$ is the
Lebesgue measure of a unit ball $B(0;1)$, and
by Lemma~\ref{ch-def-dint}, we can compute
\[
\int_{\Real^{d}} h\circ f \,dx = -S\int^{y_{\infty}+\varepsilon
}_{y_{\infty}} h'(y)(y-y_{\infty})^{d}\,dy.
\]
The assumption \hyperlink{aD-integr-inffin}{(D.2)}  implies that
$\int_{\Real^{d}} h\circ g \,dx \ge\int_{\Real^{d}} h \circ f \,dx =
\infty$,
which proves the statement.
\end{pf}
%
%%--------------------------------------------------------------------------------------------------
%
\begin{lem}
\label{ch-def-dintegr}
Let $h$ be a decreasing transformation.
Then, for any convex function $g$ such that $h \circ g$ belongs
to the decreasing model $\Model(h)$, we have
$\int_{\Real^{d}}[h | \log h | ]\circ g \,dx<\infty$.
\end{lem}
%
%%--------------------------------------------------------------------------------------------------
\begin{pf}
%Since $g$ is a proper convex function, $\inf g>-\infty$ and by Lemma
%1 = \int h \circ g \,dx = -\int_{\inf(g)}^{+\infty}h'(y)\mu(\lev_y g)\,dy.
%The function $-[h \log h](y)$ is decreasing for $y$ large enough.
%By Lemma~\ref{ch-def-dlev} the level sets $\lev_y g$ are
%bounded, and since $h \circ g\in\Model(h)$ we have $\inf g> y_{
%Therefore the integral is finite if and only if the integral:
%&\int_{(\lev_a g)^c} [ h\log h ]\circ g dx
%is finite. Choosing $a$ large enough and using Lemma~\ref{ch-def-dint}
%we obtain:
%0\ge\int_{(\lev_a g)^c} [h \log h]\circ g dx
%&= -\int_a^{+\infty}h'(y)(\log h(y) + 1)\mu(\lev_y g)\,dy\\
%& \ge-1-\int_a^{+\infty}h'(y)\log h(y)\mu(\lev_y g)\,dy .
%By Lemma~\ref{ch-def-dlev} we have $\mu(\lev_y g)=O(y^d)$.
%Let us choose
%$\alpha>0$ and $\delta>0$ such that $(d+\varepsilon)(1-\alpha)=d+
%we have:
%&\lim_{x\to0} x^{\alpha}\log x = 0
%&&\Rightarrow&&\lim_{y\to+\infty} h(y)^{\alpha}\log h(y) =0\\
%&\lim_{y\to+\infty}(\frac{h(y)}{y^{-(d+\varepsilon)}})^{1-
%&&\Rightarrow&&\lim_{y\to+\infty}\frac{h'(y)h(y)^{-\alpha}}{y^{-(d+1+
%which implies
%&h'(y)\log h(y)\mu(\lev_y g) = o(y^{-(1+\delta)}),   y\to+\infty,
%and thus the integral exists.
By assumption \hyperlink{aD-integr-zero}{(D.1)}, the function $-[h\log h](y)$ is
decreasing to zero as $y\to+\infty$ and we have
$0 < -[h\log h](y) < Cy^{-d-\alpha'}$
for $C$ large enough and $\alpha'\in(0,\alpha)$ as $y \rightarrow
+\infty$.

By Lemma~\ref{ch-def-dlev}, the level sets $\lev_y g$ are bounded
and since $h\circ g\in\Model(h)$, we have $\inf g > y_{\infty}$.
Therefore, the integral exists if and only if the integral
\[
\int_{(\lev_a g)^c} [h\log h]\circ g \,dx > -\infty
\]
for some $a>y_{\infty}$.
Choosing $a$ large enough and using Lemma~\ref{ch-def-dint} for the
decreasing transformation $h_{1}(y)=y^{-d-\alpha'}$, we obtain
\begin{eqnarray*}
0&\ge&\int_{(\lev_a g)^c} [h\log h]\circ g \,dx\\
&\ge&-C\int_{(\lev_a g)^c} h_{1}\circ g \,dx
\ge C\int_{a}^{+\infty}h_{1}'(y)\mu[\lev_{y} g]\,dy\\
&=& -C(d+\alpha')\int_{a}^{+\infty}y^{-d-\alpha'-1}\mu[\lev_{y} g]\,dy.
\end{eqnarray*}
By Lemma S.A.3,
%S.\ref{supp:cvx-gen-levbdd},
we have $\mu(\lev_y g)=O(y^d)$
and therefore the last integral is finite.
\end{pf}
%
%%--------------------------------------------------------------------------------------------------
\begin{lem}
\label{ch-def-dcompact}
Let $h$ be a decreasing transformation and suppose that
$K\subset\Real^{d}$ is a compact set.
There then exists a closed proper convex function $g\in\ModelC(h)$
such that $g<y_{0}$ on $K$.
\end{lem}
%
%%--------------------------------------------------------------------------------------------------
\begin{pf}
Let $B$ be a ball such that $K\subset B$.
Let $c$ be such that $h(c) = 1/\mu[B]$.
The function $g \equiv c + \delta(\cdot\mathop|B)$ then belongs to
$\ModelC(h)$.
\end{pf}

%s3.2 ###
\subsection{Proofs for existence results}
\label{subsec:ProofsExi}
Before giving proofs of Theorems~\ref{mle-iexi-main} and~\ref{mle-dexi-main},
we establish two auxiliary lemmas.
A set of points $x=\{x_i\}_{i=1}^n$ in $\Real^d$ is in
\textit{general position} if, for any subset $x'\subseteq x$ of size
$d+1$, the
Lebesgue measure of $\operatorname{conv}(x')$
is not zero. It follows from \citet{MR0331643}
that the observations $X$ are in general position with probability $1$ if
$X_1 , \ldots, X_n$ are i.i.d. $p_0 \in\Model(h)$.
%%--------------------------------------------------------------------------------------------------
%$X_1 , \ldots, X_n$ are i.i.d. $p_0 \in\Model(h)$.
%for a monotone transformation $h$,
%then the observations
%%--------------------------------------------------------------------------------------------------
Thus, we may assume in the following that our observations are in
general position
for every $n$. For an increasing model, we also assume that all
$X_{i}$ belong to $\Real_+^{d}$.
%This assumption holds with probability $1$ since
%$\mu[\RealPc^{d}\setminus\Real_+^{d}]=0$.

If an MLE for the model $\Model(h)$ exists, then it maximizes the functional
\[
\mathbb{L}_n g \equiv\int(\log h)\circ g \,d\mathbb{P}_n
\]
over $g\in\ModelC(h)$,
where the last integral is over ${\overline{\mathbb{R}}}{}^{d}_{+}$ for increasing $h$
and over $\Real^{d}$ for decreasing models.
The theorem below determines the form of the MLE for an increasing model.
We write $\ev_x f = (f(x_1), \ldots, f(x_n))$,
$x = (x_1 , \ldots, x_n)$ with $x_i \in\Real^d$.
%%--------------------------------------------------------------------------------------------------
%
\begin{lem}[(S.1.7)]%[S.\ref{supp:ch-mle-iup}]
\label{ch-mle-iup}
Consider an increasing transformation $h$.
For any convex function $g$ with $\dom g = {\overline{\mathbb{R}}}{}^{d}_{+}$ such that
$\int_{{\overline{\mathbb{R}}}{}^{d}_{+}}h \circ g \,dx \le1$
and $\mathbb{L}_{n} g>-\infty$, there exists $\tilde g\in\ModelC(h)$
such that $\tilde g\ge g$ and $\mathbb{L}_{n}\tilde g\ge\mathbb
{L}_{n} g$.
The function\vspace*{1pt} $\tilde g$ can be chosen as a minimal element in $\ev
^{-1}_{X} \tilde p$, where
$\tilde p = \ev_{X} \tilde g$.
\end{lem}
%
%%--------------------------------------------------------------------------------------------------
%
\begin{theorem}[(S.1.8)]%[S.\ref{supp:ch-mle-iform}]
\label{ch-mle-iform}
If an MLE $\hat{g}_0$ exists for the increasing model $\Model
(h)$, then there exists an
MLE $\hat{g}_{1}$ which is a minimal element in
$\ev^{-1}_X q$, where $q=\ev_{X} \hat{g}_{0}$.
In other words, $\hat{g}_{1}$ is a polyhedral convex function
such that $\dom g_{1}={\overline{\mathbb{R}}}{}^{d}_{+}$
and the interior of each facet contains at least one element of $X$.
If $h$ is strictly increasing on $[y_{0},y_{\infty}]$, then
$\hat{g}_{0}(x)=\hat{g}_{1}(x)$
for all $x$ such that $\hat{g}_{0}(x)>y_{0}$
and thus defines the same density from $\Model(h)$.
\end{theorem}
%
%%--------------------------------------------------------------------------------------------------

Here are the corresponding results for decreasing transformations $h$.
%
%%--------------------------------------------------------------------------------------------------
%
\begin{lem}[(S.1.9)]%[S.\ref{supp:ch-mle-dup}]
\label{ch-mle-dup}
Consider a decreasing transformation $h$.
For any convex function $g$ such that
$\int_{\Real^{d}}h\circ g \,dx \le1$
and $\mathbb{L}_{n} g>-\infty$, there exists
$\tilde g \in\ModelC(h)$ such that
$\tilde g\le g$ and $\mathbb{L}_{n}\tilde g\ge\mathbb{L}_{n} g$.
The function\vspace*{1pt} $\tilde g$ can be chosen as the maximal element in
$\ev^{-1}_{X} \tilde q$, where $\tilde q = \ev_{X} \tilde g$.
\end{lem}
%
%%--------------------------------------------------------------------------------------------------
%
\begin{theorem}[(S.1.10)]%[S.\ref{supp:ch-mle-dform}]
\label{ch-mle-dform}
If the MLE $\hat{g}_{0}$ exists for the decreasing model $\Model(h)$,
then there exists another MLE $\hat{g}_{1}$ which is the maximal
element in
$\ev^{-1}_X q$, where
$q=\ev_{X} \hat{g}_{0}$.
In other words, $\hat{g}_{1}$ is a polyhedral convex function with the
set of knots $K_n\subseteq X$ and domain $\dom\hat{g}_1=
\operatorname{conv}(X)$.
If $h$ is strictly decreasing on $[y_{\infty},y_{0}]$,
then $\hat{g}_{0}(x)=\hat{g}_1 (x) $.
\end{theorem}
%
%%--------------------------------------------------------------------------------------------------

The bounds provided by the following key lemma are the remaining
preparatory work
for proving existence of the MLE in the case of increasing transformations.\vadjust{\goodbreak}

For an increasing model $\Model(h)$, let us denote by
$\mathcal{N}(h,X,\varepsilon)$, for
$\varepsilon>-\infty$, the family of all convex functions $g\in
\ModelC(h)$
such that $g$ is a minimal element in
$\ev_{X}^{-1} q$, where
$q=\ev_{X} g$ and $\mathbb{L}_{n} g \ge\varepsilon$.
By Lemma S.1.5, %S.\ref{supp:ch-def-icompact},
the family $\mathcal
{N}(h,X,\varepsilon)$ is not empty
for $\varepsilon>-\infty$ small enough.
By construction, for $g\in\mathcal{N}(h,X,\varepsilon)$, we have
$g(X_{i})>y_{0}$ for $X_{i}\in X$.
%
%%--------------------------------------------------------------------------------------------------
%
\begin{lem}
\label{mle-iexi-bdd}
There exist constants $c(x,X,\varepsilon)$ and $C(x,X,\varepsilon
)<y_{\infty}$ which depend only on
$x\in{\overline{\mathbb{R}}}{}^{d}_{+}$, the observations $X$ and $\varepsilon$,
such that for any $g\in\mathcal{N}(h,X,\varepsilon)$, we have
\[
c(x,X,\varepsilon)\le g(x) \le C(x,X,\varepsilon).
\]
\end{lem}
%
%%--------------------------------------------------------------------------------------------------
\begin{pf}
By Lemma S.1.1, %S.\ref{supp:ch-def-ilev},
we have
$h \circ g(X_{i}) \le\frac{d!}{d^{d}\pnorm{X_{i}}}$,
which gives the upper bounds $C(X_{i},X,\varepsilon)$.
By assumption, we have
\[
\bigl(\max h \circ g(X_{i})\bigr)^{n-1}\min h\circ g(X_{i})\ge\prod
h\circ g(X_{i})\ge e^{n\varepsilon}
\]
and therefore
\[
\min h\circ g(X_{i}) \ge\frac{e^{n\varepsilon}}{h(\max
C(X_{i},X,\varepsilon))^{n-1}},
\]
which gives the uniform lower bound $c(X_{i},X,\varepsilon)$ for all
$X_{i}\in X$.
Since, by Lem\-ma~S.1.1,
%S.\ref{supp:ch-def-ilev},
$g(0)\ge g(X_{i})$, we also
obtain $c(0,X,\varepsilon)$.

We now prove that there exist $C(0,X,\varepsilon)$.
Let $l$ be a linear function which defines any facet of $g$ for which
$0$ is an element.
By Lemma S.A.15, %S.\ref{supp:cvx-poly-mine},
there exists $X_{a}\in X$ which
belongs to this facet.
Then, $g(0)=l(0)$ and $g(X_{a})=l(X_{a})$. Let us denote by $S$ the simplex
$\{l=l(X_{a})\}\cap{\overline{\mathbb{R}}}{}^{d}_{+}$, by
$S^{*}$ the simplex $\{l\ge l(X_{a})\}\cap{\overline{\mathbb{R}}}{}^{d}_{+}$
and by $l'$ the linear function which is equal to
$c\equiv\min c(X_{i},X,\varepsilon)$ on $S$ and to $g(0)$ at $0$.
By the inequality of arithmetic and geometric means (as in the proof of
Lemma S.1.1), %S.\ref{supp:ch-def-ilev}),
we have
$\mu[S^{*}]\ge\frac{d^{d}\pnorm{X_{a}}}{d!}$.
We also have, for $l\ge l'$,
$1 = \int_{{\overline{\mathbb{R}}}{}^{d}_{+}} h \circ g \,dx\ge\int_{S^{*}} h \circ l' \,dx$.
By Lemma S.1.2, %S.\ref{supp:ch-def-iint},
\begin{eqnarray*}
\int_{S^{*}} h \circ l' \,dx
&=& \mu[S^{*}]\int_{c}^{g(0)} h'(y)\biggl(\frac{g(0)-y}{g(0)-c}
\biggr)^{d}\,dy\\
&\ge&\frac{d^{d}\pnorm{X_{a}}}{d!}\int_{c}^{y_{\infty}} h'(y)1\{
y\le g(0)\}\biggl(\frac{g(0)-y}{g(0)-c}\biggr)^{d}\,dy.
\end{eqnarray*}
Consider the function $T(s)$ defined as
\[
T(s) = \frac{d^{d}\pnorm{X_{a}}}{d!}\int_{c}^{y_{\infty}} h'(y)1\{
y\le s\}\biggl(\frac{s-y}{s-c}\biggr)^{d}\,dy.
\]
If $y_{\infty}=+\infty$, then, for a fixed $y\in(c,+\infty)$, we have
\[
h'(y)1\{y\le s\}\biggl(\frac{s-y}{s-c}\biggr)^{d}\uparrow h'(y)\qquad\mbox
{as }s\to y_{\infty}
\]
and, by monotone convergence, we have
\[
T(s)\uparrow\int_{c}^{y_{\infty}} h'(y)\,dy
= +\infty\qquad\mbox{as }s\to y_{\infty}.
\]
If $y_{\infty}<+\infty$, then for a fixed $y\in(c,y_{\infty}]$, we have
\[
h'(y)1\{y\le s\}\biggl(\frac{s-y}{s-c}\biggr)^{d}\uparrow h'(y)
\biggl(\frac{y_{\infty}-y}{y_{\infty}-c}\biggr)^{d}\qquad\mbox{as }s\to
y_{\infty}
\]
and, by monotone convergence, we have
\[
T(s)\uparrow\int_{c}^{y_{\infty}} h'(y)\biggl(\frac{y_{\infty
}-y}{y_{\infty}-c}\biggr)^{d}\,dy
= +\infty\qquad\mbox{as }s\to y_{\infty},
\]
by assumption \hyperlink{aI-integr-inf}{(I.2)}.

Thus, there exists $s_{0}\in(c,y_{\infty})$ such that $T(s_{0})>1$.
This implies that \mbox{$g(0)<s_{0}$}. Since $s_{0}$ depends only on $X_{a}$
and $\min c(X_{i},X,\varepsilon)$, this gives an upper bound
$C(0,X,\varepsilon)$.

By Lemma S.1.1, %S.\ref{supp:ch-def-ilev},
for any $x_{0}\in{\overline{\mathbb{R}}}{}^{d}_{+}$, we
can set $
C(x_{0},X,\varepsilon) = C(0,X,\varepsilon)$.
Let $l(x) = a^{T}x + l(0)$ be a linear function which defines the facet
of $g$ to which $x$ belongs.
By Lemma S.A.15, %S.\ref{supp:cvx-poly-mine},
there exists $X_{a}\in X$ which
belongs to this facet
and thus $l(X_{a}) = g(X_{a})$. By Lemma S.1.1, %S.\ref{supp:ch-def-ilev},
we have $a_{k}<0$ for all $k$
and, by definition, $l(0)\le g(0)$. We have
\[
c(X_{a},X,\varepsilon)\le g(X_{a}) = l(X_{a}) = a^{T}X_{a} + l(0)\le
a^{T}X_{a} + g(0),
\]
therefore
\[
a_{k} \ge\frac{c(X_{a},X,\varepsilon)-C(0,X,\varepsilon
)}{(X_{a})_{k}} \quad\mbox{and}\quad
l(0) \ge c(X_{a},X,\varepsilon).
\]
Now,
\[
g(x_{0}) = l(x_{0}) \ge\frac{c(X_{a},X,\varepsilon
)-C(0,X,\varepsilon)}{(X_{a})_{k}}(x_{0})_{k} + c(X_{a},X,\varepsilon).
\]
Since we have only a finite number of possible choices for $X_{a}$,
we have obtained $c(x_{0},X,\varepsilon)$, which completes the proof.
\end{pf}

We are now ready for the proof of Theorem~\ref{mle-iexi-main}.
\begin{pf*}{Proof of Theorem~\ref{mle-iexi-main}}
By Lemma S.1.5,
%S.\ref{supp:ch-def-icompact},
there exists $\varepsilon$
small enough
such that the family $\mathcal{N}(h,X,\varepsilon)$ is not empty.
Clearly, we can restrict MLE candidates $\hat{g}$
to functions in the family $\mathcal{N}(h,X,\varepsilon)$.
The set $N = \ev_{X} \mathcal{N}(h,X,\varepsilon)$ is bounded,
by Lemma~\ref{mle-iexi-bdd}. Let us denote by $q^{*}$ a point in the
closure $\bar N$ of
$N$ which maximizes the continuous function
\[
L_{n}(q) = \frac{1}{n}\sum_{i=1}^{n} \log h(q_{i}).
\]
Since $q^{*}\in\bar N$, there exists a sequence of functions
$g_{k}\in\mathcal{N}(h,X,\varepsilon)$ such that $\ev_{X} g_{k}$
converges to\vadjust{\goodbreak}
$q^{*}$. By Theorem 10.9 in \citet{MR0274683} and Lemma~\ref
{mle-iexi-bdd},
there exists a finite convex function $g^{*}$ on ${\overline{\mathbb{R}}}{}^{d}_{+}$
such that some subsequence $g_{l}$ converges pointwise to $g^{*}$.
Therefore, we have $\ev_{X} g^{*} = q^{*}$.
Since $X\subset\Real_+^{d}$, we can assume that $g^{*}$ is closed.
By Fatou's lemma, we have
$\int_{{\overline{\mathbb{R}}}{}^{d}_{+}} h\circ g^{*}\,dx \le1$.
By Lemma~\ref{ch-mle-iup}, there exists $g\in\ModelC(h)$
such that $g\ge g^{*}$ and
$\mathbb{L}_{n}g \ge\mathbb{L}_{n}g^{*} = L_{n}(q^{*})$.
By assumption, this implies that $\mathbb{L}_{n}g = \mathbb{L}_{n}g^{*}$.
Hence, $g$ is the MLE.
Finally, we have to add the ``almost surely'' clause since
we have assumed that the points $X_{i}$ belong to $\Real_+^{d}$.
\end{pf*}

Before proving existence of the MLE for a decreasing transformation
family, we
need two lemmas.
%
%%--------------------------------------------------------------------------------------------------
%
\begin{lem}[(S.1.11)]%[S.\ref{supp:mle-dexi-bound}]
\label{mle-dexi-bound}
Consider a decreasing model $\Model(h)$.
Let $\{g_{k}\}$ be a sequence of convex functions from $\ModelC(h)$
and let $\{n_{k}\}$ be a nondecreasing sequence of positive integers
$n_{k}\ge n_{d}$ such that for some $\varepsilon>-\infty$ and $\rho
>0$, the following is true:
\begin{enumerate}
\item$\mathbb{L}_{n_{k}}g_{k}\ge\varepsilon$;
\item if $\mu[\lev_{a_{k}} g_{k}]=\rho$ for some $a_{k}$,
then $\mathbb{P}_{n_{k}}[\lev_{a_{k}} g_{k}]<d/n_{d}$.
\end{enumerate}
There then exists $m>y_{\infty}$ such that $g_{k}\ge m$ for all $k$.
\end{lem}
%
%%--------------------------------------------------------------------------------------------------

For a decreasing model $\Model(h)$,
let us denote by $\mathcal{N}(h,X,\varepsilon)$ for $\varepsilon
>-\infty$ the family of all convex functions $g \in\ModelC(h)$
such that $g$ is a maximal element in $\ev_{X}^{-1} q$, where $q=\ev
_{X} g$
and $\mathbb{L}_{n} g\ge\varepsilon$.
By Lemma~\ref{ch-def-dcompact}, the family $\mathcal
{N}(h,X,\varepsilon)$ is not empty for $\varepsilon>-\infty$
small enough.
By construction, for $g\in\mathcal{N}(h,X,\varepsilon)$, we have
$g(X_{i})<y_{0}$ for $X_{i}\in X$.
%%--------------------------------------------------------------------------------------------------
%
\begin{lem}
\label{mle-dexi-bdd}
For given observations $X=(X_{1},\ldots,X_{n})$ such that $n\ge n_{d}$,
there exist constants $m>y_{\infty}$ and $M$ which depend only on observations
$X$ and $\varepsilon$ such that for any $g\in\mathcal
{N}(h,X,\varepsilon)$, we have
$m\le g(x) \le M$ on $\operatorname{conv}(X)$.
\end{lem}
%
%%--------------------------------------------------------------------------------------------------
\begin{pf}
Since, by assumption, the points $X$ are in general position, there exists
$\rho>0$ such that for any $d$-dimensional simplex $S$ with vertices
from $X$, we have
$\mu[S]\ge\rho$. Then, any convex set $C\subseteq
\operatorname{conv}(X)$ such that
$\mu[C]=\rho$ cannot contain more than $d$ points from $X$.
Therefore, we have $\mathbb{P}_{n}[C]\le d/n\le d/n_{d}$.

An arbitrary sequence of functions $\{g_{k}\}$ from $\mathcal
{N}(h,X,\varepsilon)$
satisfies the conditions of Lemma~\ref{mle-dexi-bound} with
$n_{k}\equiv n$ and
the same $\varepsilon$ and $\rho$ constructed above.
Therefore, the sequence $\{g_{k}\}$ is bounded below by some constant
greater than $y_{\infty}$.
Thus, the family of functions $\mathcal{N}(h,X,\varepsilon)$ is
uniformly bounded below by some $m>y_{\infty}$.

Consider any $g\in\mathcal{N}(h,X,\varepsilon)$.
Let $M_{g}$ be the supremum of $g$ on $\dom h$.
By Theorem 32.2 in \citet{MR0274683}, the supremum is obtained at
some $X_{M}\in X$
and therefore $M_{g}<y_{0}$.
Let $m_{g}$ be the minimum of $g$ on $X$. We have
$h(m_{g})^{n-1}h(M_{g}) \ge e^{n\varepsilon}$ and
\[
h(M_{g}) \ge\frac{e^{n\varepsilon}}{h(m_{g})^{n-1}}\ge\frac
{e^{n\varepsilon}}{h(m)^{n-1}}.
\]
Thus, we have obtained an upper bound $M$ which depends only on $m$ and~$X$.
\end{pf}

%Now we can prove Theorem~\ref{mle-dexi-main}.%~\ref{}.
We are now ready for the proof of Theorem~\ref{mle-dexi-main}.
\begin{pf*}{Proof of Theorem~\ref{mle-dexi-main}}
By Lemma~\ref{ch-def-dcompact}, there exists $\varepsilon$ small enough
such that the family $\mathcal{N}(h,X,\varepsilon)$ is not empty.
Clearly, we can restrict MLE candidates to the functions in the family
$\mathcal{N}(h,X,\varepsilon)$. The set $N = \ev_{X} \mathcal
{N}(h,\break X,\varepsilon)$
is bounded, by Lemma~\ref{mle-dexi-bdd}.
Let us denote by $q^{*}$ a point in the
closure $\bar N$ of $N$ which maximizes the continuous function
\[
L_{n}(q) = \frac{1}{n}\sum_{i=1}^{n} \log h(q_{i}).
\]
Since $q^{*} \in\bar N$, there exists a sequence of functions
$g_{k}\in\mathcal{N}(h,X,\varepsilon)$
such that $\ev_{X} g_{k}$ converges to $q^{*}$.
By Lemma~\ref{mle-dexi-bdd}, the functions $f_{k}=\sup_{l\ge k}g_{l}$
are finite
convex functions on $\operatorname{conv}(X)$
and the sequence $\{f_{k}(x)\}$ is monotone decreasing for each
$x\in\operatorname{conv}(X)$ and bounded below.
Therefore, $f_{k}\downarrow g^{*}$ for some convex function $g^{*}$
and, by construction, $\ev_{X} g^{*} = q^{*}$. We have
\[
\int_{\Real^{d}} h\circ f_{k}\,dx \le\int_{\Real^{d}} h \circ
g_{k}\,dx = 1
\]
and thus, by Fatou's lemma,
$\int_{\Real^{d}} h\circ g^{*}\,dx\le1$.
By Lemma~\ref{ch-mle-dup}, there exists $g\in\ModelC(h)$
such that $g\le g^{*}$ and $\mathbb{L}_{n}g\ge\mathbb{L}_{n}g^{*}
= L_{n}(q^{*})$. By assumption, this implies that $\mathbb{L}_{n}g =
\mathbb{L}_{n}g^{*}$.
Thus, the function $g$ is the MLE.

Finally, we have to add the ``almost surely'' clause since we assumed
that the points $X_{i}$ are in general position.
\end{pf*}

%s3.3 ###
\subsection{Proofs for consistency results}
We begin with proofs for some technical results which we will use in
the consistency
arguments for both increasing and decreasing models.
The main argument for proving Hellinger consistency proceeds along the
lines of the proof given in the case of $d=1$ by
\citet{PalWoodroofeMeyer07} and in the log-concave case for $d>1$ by
\citet{SchuhmacherDuembgen2010}.
%
%%--------------------------------------------------------------------------------------------------
%
\begin{lem}[(S.1.12)]%[S.\ref{supp:mle-cons-hlgr}]
\label{mle-cons-hlgr}
Consider a monotone model $\Model(h)$. Suppose that the true density
$h\circ g_{0}$ and the sequence of MLEs $\{\hat g_{n}\}$ have the
following properties:
\[
\int(h |{\log h}|)\circ g_{0}(x)\,dx <\infty
\]
and
\[
\int\log[\varepsilon+h\circ\hat g_n(x)]\,d\bigl(\mathbb{P}_n(x) -
P_{0}(x)\bigr)\to_{a.s.} 0
\]
for $\varepsilon>0$ small enough.
The sequence of the MLEs is then Hellinger consistent:
$H(h\circ\hat g_{n},h\circ g_{0})\to_{a.s.} 0$.
\end{lem}
%
%%--------------------------------------------------------------------------------------------------

%%--------------------------------------------------------------------------------------------------
The next lemma allows us to obtain pointwise consistency once Hellinger
consistency is
proved.
\begin{lem}
\label{mle-cons-pw}
Suppose that, for a monotone model $\Model(h)$, a sequence of
MLEs $\hat g_{n}$ is Hellinger consistent.
The sequence $\hat g_{n}$ is then pointwise consistent.
In other words, $\hat g_{n}(x)\to_{a.s.} g_{0}(x)$ for $x\in\ri(\dom g_{0})$
and convergence is uniform on compacta.
\end{lem}
%
%%--------------------------------------------------------------------------------------------------
\begin{pf}
Let us denote by $L^0_a$ and $L^k_a$ the sublevel sets
$L^0_a = \lev_a g_{0}$ and
$L^n_a = \lev_a \hat g_n$, respectively.
Consider $\Omega_0$ such that $\Pr[\Omega_0]=1$ and $H^{2}(h\circ
\hat
g_n^{\omega},h\circ g_0)\to0$, where $\hat g_{n}^{\omega}$ is the
MLE for $\omega\in\Omega_{0}$. For all $\omega\in\Omega_0$, we have
\begin{eqnarray*}
\int\bigl[\sqrt{h}\circ g_0 - \sqrt{h}\circ\hat g_n\bigr]^2 \,dx &\ge&
\int_{L^0_a\setminus L^n_{a+\varepsilon}} \bigl[\sqrt{h}\circ g_0 - \sqrt
{h}\circ
\hat g_n\bigr]^2 \,dx\\
&\ge&\bigl(\sqrt{h}(a)-\sqrt{h}(a+\varepsilon)\bigr)^2\mu(L^0_a\setminus
L^n_{a+\varepsilon})\\
&\to&0
\end{eqnarray*}
and, by Lemma S.A.2, %S.\ref{supp:cvx-gen-setlim},
we have
$\liminf\ri(L^0_a\cap L^n_{a+\varepsilon})= \ri(L^0_a)$.
Therefore, $\limsup\hat g_n(x) < a+\varepsilon$ for
$x\in\ri(L^{0}_{a})$. Since $a$ and $\varepsilon$ are arbitrary, we
have $\limsup\hat g_n
\le g_{0}$ on $\ri(\dom g_{0})$.

On the other hand, we have
\begin{eqnarray*}
\int\bigl[\sqrt{h}\circ g_0 - \sqrt{h}\circ\hat g_n\bigr]^2 \,dx &\ge&
\int_{L^n_{a-\varepsilon}\setminus L^0_{a}} \bigl[\sqrt{h}\circ g_0 -
\sqrt{h}\circ\hat g_n\bigr]^2 \,dx\\
&\ge&\bigl(\sqrt{h}(a-\varepsilon)-\sqrt{h}(a)\bigr)^2\mu(L^n_{a-\varepsilon
}\setminus L^0_{a})\\
&\to&0
\end{eqnarray*}
and by Lemma S.A.2, %S.\ref{supp:cvx-gen-setlim},
we have
$\limsup\cl(L^n_{a-\varepsilon}\cup L^0_{a})= \cl(L^0_a)$.
Therefore, $\liminf\hat g_n(x)> a-\varepsilon$ for $x$ such that
$g_{0}(x)\ge a$. Since $a$ and $\varepsilon$ are arbitrary, we have
$\liminf\hat g_n
\ge g_{0}$ on $\dom g_{0}$.

Thus, $\hat g_n\to g_{0}$ almost surely on $\ri(\dom g_{0})$. By
Theorem 10.8 in \citet{MR0274683}, convergence is uniform on
compacta $K\subset\ri(\RealP^{d})$.
\end{pf}

%%--------------------------------------------------------------------------------------------------
We need a general property of the bracketing entropy numbers.
\begin{lem}[(S.1.13)]%[S.\ref{supp:mle-cons-GC}]
\label{mle-cons-GC}
Let $\mathcal A$ be a class of sets in $\Real^d$ such that class
$\mathcal{A}\cap[-a,a]^{d}$ has
finite bracketing entropy with respect to Lebesgue measure $\mu$ for
any $a$ large enough:\vadjust{\goodbreak}
$\log N_{[]}(\varepsilon,{\mathcal A}\cap[-a,a]^{d},L_1(\mu
))<+\infty$
for every $\varepsilon>0$. Then, for any Lebesgue absolutely
continuous probability measure $P$ with bounded density, we have that
$\mathcal{A}$ is a Glivenko--Cantelli class:
$\|\mathbb{P}_{n}-P\|_{\mathcal A}\to_{a.s.} 0$.
\end{lem}
%
%%--------------------------------------------------------------------------------------------------

By Lemma S.1.1, %S.\ref{supp:ch-def-ilev},
we have $\ri(\RealP^{d})
\subseteq\dom g_{0}$. Thus, Theorem \ref
{thm:IncHellingerConsistencyOfMLE} and Lem\-ma~\ref{mle-cons-pw} imply
Theorem~\ref{thm:IncPointwiseAndUniformConsistency}.
%%%%%%%%%%%%%%%%%%%%%%%%%%%%%%%%%%%%%%%%%%%%%%%%%%%%%%%%%%%%%%%%%%%%%%%%%%%%%%%%%%%%%%%%%%%%%%%%%%%%

Finally, we prove consistency for decreasing models.
We need a general property of convex sets.
%%--------------------------------------------------------------------------------------------------
%
\begin{lem}
\label{mle-dcons-GC}
Let $\mathcal A$ be the class of closed convex sets $A$ in $\Real^d$
and let $P$ be a
Lebesgue absolutely continuous probability measure
with bounded density. Then,
$\|\mathbb{P}_{n}-P\|_{\mathcal A}\to_{a.s.} 0$.
\end{lem}
%
%%--------------------------------------------------------------------------------------------------
\begin{pf}
Let $D$ be a convex compact set. By Theorem 8.4.2 in \citet
{MR1720712}, the class $\mathcal{A}\cap D$ has a finite set of
$\varepsilon$-brackets.
Since the class $\mathcal{A}$ is invariant under rescaling, the result
follows from Lemma~\ref{mle-cons-GC}.
\end{pf}
%
%%--------------------------------------------------------------------------------------------------
%
\begin{lem}
\label{mle-dcons-bdd}
For a decreasing model $\Model(h)$, the sequence of MLEs $\hat g_{n}$
is almost surely uniformly bounded below.
\end{lem}
%
%%--------------------------------------------------------------------------------------------------
\begin{pf}
We will apply Lemma~\ref{mle-dexi-bound} to the sequences $\hat g_{n}$
and $\{n\}$. By the strong law of large numbers and Lemma \ref
{ch-def-dintegr}, we have
\[
\mathbb{L}_{n} \hat g_{n} \ge\mathbb{L}_{n} g_{0}\to_{\mathrm{a.s.}} \int
[h\log h]\circ g_{0} \,dx > -\infty.
\]
Therefore, the sequence $\{\mathbb{L}_{n}\hat g_{n}\}$ is bounded away
from $-\infty$ and the first condition of Lemma~\ref{mle-dexi-bound}
is true.

Choose some $a\in(0,d/n_{d})$. Then, for any set $S$ such that $\mu
[S]=\rho\equiv a/h(\min g_{0})$, where $\min g_{0}$ is attained by
Lemma~\ref{ch-def-dlev}, we have
\[
P[S] = \int_{S} h\circ g_{0}\,dx \le\mu[S]h(\min g_{0}) = a < d/n_{d}.
\]
Now, let $A_{n}=\lev_{a_{n}}\hat g_{n}$ be sets such that $\mu
[A_{n}]=\rho$. Then, by Lemma~\ref{mle-dcons-GC}, we have
\[
|\mathbb{P}_{n}[A_{n}] - P[A_{n}]|\le\|\mathbb{P}_{n}-P\|_{\mathcal
A}\to_{\mathrm{a.s}.} 0,
\]
which implies that $\mathbb{P}_{n}[A_{n}]<d/n_{d}$ almost surely for
$n$ large enough. Therefore, the second condition of Lemma \ref
{mle-dexi-bound} is true and is applicable to the sequence $\hat g_{n}$
almost surely.
\end{pf}
%
%%--------------------------------------------------------------------------------------------------
%
\begin{pf*}{Proof of Theorem~\ref{thm:DecHellingerConsistencyOfMLE}}
By Lemmas~\ref{ch-def-dintegr} and~\ref{mle-cons-hlgr}, it is
enough to show that
\[
\int\log[\varepsilon+h\circ\hat g_n(x)]\,d\bigl(\mathbb{P}_n(x) -
P_{0}(x)\bigr)\to_{\mathrm{a.s.}} 0.
\]
By Lemma~\ref{mle-dcons-bdd}, we have $\inf\hat g_{n}\ge A$ for some
$A>y_{\infty}$.
Therefore, by Lemma~\ref{ch-def-dint} applied to the decreasing
transformation $\log[\varepsilon+h(y)] - \log\varepsilon$, it
follows that
\begin{eqnarray*}
&&\int\log[\varepsilon+h\circ\hat g_n(x)]\,d\bigl(\mathbb{P}_n(x) -
P_{0}(x)\bigr)\\
&&\qquad=\int_{A}^{+\infty}\biggl[\frac{-h'(z)}{\varepsilon+h(z)}
\biggr](\mathbb{P}_n
- P_{0})(\lev_z \hat g_n)\,dz\\
&&\qquad\le\|\mathbb{P}_n - P_{0}\|_{\mathcal A}\int_{A}^{+\infty}
\biggl[\frac{-h'(z)}{\varepsilon+h(z)}\biggr]\,dz\\
&&\qquad= \|\mathbb{P}_n - P_{0}\|_{\mathcal A}\log\biggl[\frac{\varepsilon
+h(A)}{\varepsilon}\biggr]\to_{\mathrm{a.s.}} 0,
\end{eqnarray*}
where the last limit follows from Lemma~\ref{mle-dcons-GC}.
\end{pf*}
%
%%--------------------------------------------------------------------------------------------------
%
\begin{pf*}{Proof of Theorem
\ref{thm:DecPointwiseAndUniformConsistency}}
By Lemma~\ref{mle-cons-pw}, we have $\hat g_n\to g_{0}$ almost surely
on $\ri(\dom g_{0})$. Functions $g_{0}$ and $g_{0}^{*}$ differ only on
the boundary $\partial\dom g_{0}$, which has Lebesgue measure zero, by
Lemma S.A.1. %S.\ref{supp:cvx-gen-meas}.
Since observations $X_{i}\in\ri
(\dom g_{0})$ almost surely, we have $\hat g_{n}=+\infty$ on $\partial
\dom g_{0}$ and thus $\hat g_{n}\to g_{0}^{*}$.

Now, we assume that $\dom g_{0} = \Real^{d}$. By Lemma \ref
{ch-def-dlev}, the function $g_{0}$
has bounded sublevel sets and therefore there exists $x_{0}$ where
$g_{0}$ attains its minimum $m$.
Since $h\circ g_{0}$ is density, we have $h(m)>0$ and by Lemma \ref
{ch-def-dinfval},
we have $h(m)<\infty$. Fix $\varepsilon>0$ such that
$h(m)>3\varepsilon$ and consider
$a$ such that $h(a)< \varepsilon$. The set $A = \lev_{a} g_{0}$ is
bounded and,
by continuity, $g_{0} = a$ on $\partial A$. Choose $\delta>0$ such
that $h(a-\delta)<2\varepsilon<h(m+\delta)$ and
\[
{\sup_{x\in[m,a+\delta]}} |h(x)-h(x-\delta)|\le\varepsilon.
\]
The closure $\bar A$ is compact and thus, for $n$ large enough, we
have, with probability one,
${\sup_{\bar A}}|\hat g_{n} - g_{0}|<\delta$,
which implies that
${\sup_{\bar A}}|h\circ\hat g_{n} - h\circ g_{0}|<\varepsilon$
since the range of values of $g_{0}$ on $\bar A$ is $[m,a]$.
The set $\partial A$ is compact and therefore $\hat g_{n}$ attains its
minimum $m_{n}$ on this set at some point $x_{n}$. By construction,
\[
m_{n}=\hat g_{n}(x_{n})>g_{0}(x_{n})-\delta
= a-\delta> m+\delta=g_{0}(x_{0})+\delta> \hat g_{n}(x_{0}).
\]
We have $x_{0}\in A\cap\lev_{a-\delta} \hat g_{n}$ and $\hat
g_{n}\ge m_{n}>a-\delta$ on $\partial A$.
Thus, by convexity, we have $\lev_{a-\delta}\hat g_{n}\subset A$ and
for $x\notin\bar A$, we have
\[
|h\circ\hat g_{n}(x) - h\circ g_{0}(x)|\le h\circ\hat g_{n}(x) +
h\circ g_{0}(x)<h(a-\delta) + h(a)<3\varepsilon.
\]
This shows that for any $\varepsilon>0$ small enough, we will have
\[
\|h\circ\hat g_{n}-h\circ g_{0}\|_{\infty}<3\varepsilon
\]
with probability one as $n\to\infty$. This concludes the proof.
\end{pf*}

%s3.4 ###
\subsection{Proofs for lower bound results}

%%--------------------------------------------------------------------------------------------------------------------
We will use the following lemma for computing the Hellinger distance
between a function and its local deformation.
%%--------------------------------------------------------------------------------------------------------------------
%
\begin{lem}[(S.3.1)]%[S.\ref{supp:lb-deflim}]
\label{lb-deflim}
Let $\{g_{\varepsilon}\}$ be a local deformation of the function
$g\dvtx\Real^{d}\to\Real$ at the point $x_{0}$
such that $g$ is continuous at $x_{0}$ and let the function $h\dvtx\Real
\to\Real$ be
continuously differentiable at the point $g(x_{0})$. Then, for any $r>0$,
%
%e3.2 ###
%e3.1 ###
\begin{eqnarray}
\label{lb-deflim1}
\lim_{\varepsilon\to0}\int_{\Real^{d}} | g_{\varepsilon
}(x)-g(x)|^{r}\,dx &=& 0,\\
\label{lb-deflim2}
\lim_{\varepsilon\to0}\frac{\int_{\Real^{d}} | h\circ
g_{\varepsilon}(x)-h\circ g(x)|^{r}\,dx}
{\int_{\Real^{d}} | g_{\varepsilon}(x)-g(x)|^{r}\,dx} &=& |h'\circ
g(x_{0})|^{r}.
\end{eqnarray}
\end{lem}

%%--------------------------------------------------------------------------------------------------------------------
In order to apply Corollary~\ref{lb-cor}, we need to
construct deformations so that they still belong to the class
$\mathcal{G}$. The following lemma provides a technique for
constructing such deformations.
%%--------------------------------------------------------------------------------------------------------------------
%
\begin{lem}[(S.3.2)]%[S.\ref{supp:lb-defnorm}]
\label{lb-defnorm}
Let $\{g_{\varepsilon}\}$ be a local deformation of the function
$g\dvtx\Real^{d}\to\Real$ at the point $x_{0}$ such that $g$ is
continuous at
$x_{0}$ and let the function $h\dvtx\Real\to\Real$ be continuously differentiable
at the point $g(x_{0})$ so that $h'\circ g(x_{0})\neq0$.
Then, for any fixed $\delta>0$ small enough, the deformation
$g_{\theta,\delta} = \theta g_{\delta} + (1-\theta)g$ and any
$r>0$, we have
%
%e3.4 ###
%e3.3 ###
\begin{eqnarray}
\label{lb-deflim3}
\limsup_{\theta\to0}\theta^{-r}\int_{\Real^{d}} |h \circ
g_{\theta,\delta}(x)-h\circ g(x)|^{r}\,dx &<& \infty,\\
\label{lb-deflim4}
\liminf_{\theta\to0}\theta^{-r}\int_{\Real^{d}} |h\circ
g_{\theta,\delta}(x)-h\circ g(x)|^{r}\,dx &>& 0.
\end{eqnarray}
\end{lem}

%%--------------------------------------------------------------------------------------------------------------------
Note that $g_{\theta,\delta}$ is not a local deformation of $g$.
%%--------------------------------------------------------------------------------------------------------------------
%
\begin{pf*}{Proof of Theorem~\ref{lb-point-main}}
Our statement is nontrivial only if the curvature $\curv_{x_{0}} g >
0$ or,
equivalently, there exists a positive definite $d\times d$ matrix $G$ such
that the function $g$ is locally $G$-strongly convex.
Then, by Lemma S.A.17, %S.\ref{supp:cvx-strong-equiv},
this means that there
exists a convex function
$q$ such that, in some neighborhood $O(x_{0})$ of $x_{0}$, we have
%
%e3.5 ###
\begin{equation}\label{lb-point-decomp}
g(x) = \tfrac{1}{2}(x-x_{0})^{T}G(x-x_{0}) + q(x).
\end{equation}

The plan of the proof is as follows: we introduce families of
functions $\{D_{\varepsilon}(g;x_{0},v)\}$ and $\{D_{\varepsilon
}^{*}(g;x_{0})\}$
and prove that these families are local deformations.
Using these deformations as building blocks, we construct two types of
deformations, $\{h\circ g^{+}_{\varepsilon}\}$ and $\{h \circ
g^{-}_{\varepsilon}\}$,
of the density $h\circ g$, which belong to $\Model(h)$. These
deformations represent
positive and negative changes in the value of the function $g$ at the
point $x_{0}$.
We then approximate the Hellinger distances using Lemma~\ref{lb-deflim}.
Finally, applying Corollary~\ref{lb-cor}, we obtain lower bounds which depend
on $G$. We complete the proof by taking the supremum of the obtained
lower bounds
over all $G\in\mathcal{SC}(g;x_{0})$.
Under the mild assumption of strong convexity of the function $g$, both
deformations give the same rate and structure of the constant $C(d)$.
However, it is possible to obtain a larger constant $C(d)$ for the negative
deformation if we assume that $g$ is twice differentiable. Note that,
by the
definition of $\Model(h)$, the function $g$ is a closed proper convex function.

%%--------------------------------------------------------------------------------------------------------------------
%
%f1 ###
\begin{figure}

\includegraphics{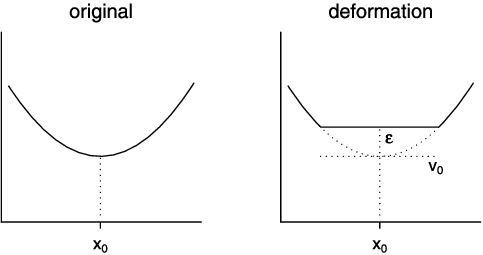}

\caption{Example of the deformation $D_{\varepsilon}(g;x_{0},v_{0})$.}
\label{fig1}
\end{figure}
%
%%--------------------------------------------------------------------------------------------------------------------
%

Let us define a function $D_{\varepsilon}(g;x_{0},v_{0})$ for a given
$\varepsilon>0$, $x_{0}\in\dom g$ and $v_{0}\in\partial g(x_{0})$ as follows:
$D_{\varepsilon}(g;x_{0},v_{0})(x) = \max(g(x), l_{0}(x)+\varepsilon)$,
where $l_{0}(x)=\langle v_{0}, x-x_{0}\rangle+ g(x_{0})$ is a support
plane to $g$ at $x_{0}$ (see Figure~\ref{fig1}).
Since $l_{0}+\varepsilon$ is a support plane to $g+\varepsilon$, we have
$g\le D_{\varepsilon}(g;x_{0},v_{0}) \le g+\varepsilon$
and thus $\dom D_{\varepsilon}(g;x_{0},v_{0}) = \dom g$.
As a maximum of two closed convex functions, $D_{\varepsilon}(g;x_{0},v_{0})$
is a closed convex function. For a given $x_{1}$, we have
$D_{\varepsilon}(g;x_{0},v_{0})(x_{1}) = g(x_{1})$ if and only if
%
%e3.6 ###
\begin{equation}\label{lb-supp-ineq-pos}
g(x_{1})-\varepsilon\ge\langle v_{0}, x_{1}-x_{0}\rangle+
g(x_{0}).
\end{equation}

We also define a function $D^{*}_{\varepsilon}(g;x_{0})$ for a given
$\varepsilon>0$
and $x_{0}\in\dom g$ as a maximal convex minorant (Appendix S.A.1)
%S.\ref{supp:cvx-mcm})
of the function $\tilde g_{\varepsilon}$, defined as
\[
\tilde g_{\varepsilon}(x) = g(x) 1_{\{x_0\}^c} (x) + \bigl(g(x_0)
-\varepsilon\bigr)1_{\{x_0\}} (x),
%= \{\begin{array}{lll} g(x), &x\neq x_{0}\\
% g(x_{0})-\varepsilon, &x=x_{0}.
%
\]
see Figure~\ref{fig2}. Both functions $D_{\varepsilon}(g;x_{0},v_{0})$ and
$D^{*}_{\varepsilon}(g;x_{0})$
are convex by construction and, as the next lemma shows, have similar
%
%f2 ###
\begin{figure}

\includegraphics{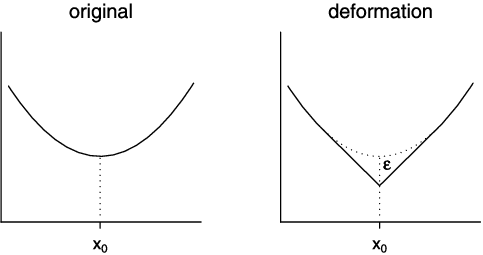}

\caption{Example of the deformation $D^{*}_{\varepsilon}(g;x_{0})$.}
\label{fig2}
\end{figure}
properties.
However, the argument for $D^{*}_{\varepsilon}(g;x_{0})$ is more complicated.
%
%%--------------------------------------------------------------------------------------------------------------------
\begin{lem}\label{lb-point-lem1}
Let $g$ be a closed proper convex function, $g^{*}$ its convex
conjugate and
$x_{0}\in\ri(\dom g)$. Then:
\begin{enumerate}
\item\label{lb-point-lem1-1} $D^{*}_{\varepsilon}(g;x_{0})$ is a
closed proper convex function such that
$ g-\varepsilon\le D^{*}_{\varepsilon}(g;x_{0}) \le g$
and $\dom D^{*}_{\varepsilon}(g;x_{0}) = \dom g$;
\item\label{lb-point-lem1-2} for a given $x_{1}\in\ri(\dom g)$, we have
$D^{*}_{\varepsilon}(g;x_{0})(x_{1})=g(x_{1})$ if and only if there
exists $v\in\partial g(x_{1})$ such that
%
%e3.7 ###
\begin{equation}\label{lb-supp-ineq-neg}
g(x_{1})+\varepsilon\le\langle v, x_{1}-x_{0}\rangle+ g(x_{0});
\end{equation}
\item\label{lb-point-lem1-3} if $v_{0}\in\partial g(x_{0})$, then
$x_{0}\in\partial g^{*}(v_{0})$ and
$D_{\varepsilon}(g;x_{0},v_{0}) = (D^{*}_{\varepsilon}(g^{*}; v_{0}))^{*}$.
\end{enumerate}
\end{lem}
%
%%--------------------------------------------------------------------------------------------------------------------
\begin{pf}
Obviously, $\tilde g_{\varepsilon} \ge g-\varepsilon$. Since
$g-\varepsilon$ is a
closed proper convex function, it is equal to the supremum of all
linear functions
$l$ such that $l\le h-\varepsilon$. Thus,
$g-\varepsilon\le D^{*}_{\varepsilon}(g;x_{0})$, which implies that
$D^{*}_{\varepsilon}(g;x_{0})$ is a proper convex function and
$\dom D^{*}_{\varepsilon}(g;x_{0})\subseteq\dom(g-\varepsilon) =
\dom g$.
By Lemma S.A.10, %S.\ref{supp:cvx-subg-prop2},
we have $D^{*}_{\varepsilon
}(g;x_{0}) \le g$ and
therefore $\dom g\subseteq\dom D_{\varepsilon}(g;x_{0})$, which
proves item~\ref{lb-point-lem1-1}.

If $v\in\partial g(x_{1})$, then $l_{v}(x) = \langle v, x-x_{1}\rangle
+ g(x_{1})$
is a support plane to $g(x)$ and $l_{v}\le g$.
If inequality~\eqref{lb-supp-ineq-neg} holds true, then $l_{v}(x)$ is
majorized by
$\tilde g_{\varepsilon}$ and we have %\[
$D_{\varepsilon}(g;x_{0})(x_{1})\le g(x_{1}) = l_{v}(x_{1}) \le
D_{\varepsilon}(g;x_{0})(x_{1})$.
On the other hand, by item~\ref{lb-point-lem1-1}, we have $x_{1}\in\ri
(\dom D_{\varepsilon}(g;x_{0}))$,
hence there exists $v\in\partial D_{\varepsilon}(g;x_{0})(x_{1})$ and
\begin{eqnarray*}
g(x)&\ge&\tilde g_{\varepsilon}(x)\ge D_{\varepsilon}(g;x_{0})(x)
\ge\langle v, x_{0}-x_{1}\rangle+ D_{\varepsilon
}(g;x_{0})(x_{1}) \\
& = & \langle v, x_{0}-x_{1}\rangle+ g(x_{1}).
\end{eqnarray*}
Therefore, $v\in\partial g(x_{1})$. In particular,
\begin{eqnarray*}
\tilde g_{\varepsilon}(x_{0})&=& g(x_{0})-\varepsilon
\ge D_{\varepsilon}(g;x_{0})(x_{0})\ge
\langle v, x_{0}-x_{1}\rangle+ D_{\varepsilon}(g;x_{0})(x_{1}) \\
&=&
\langle v, x_{0}-x_{1}\rangle+ g(x_{1}),
\end{eqnarray*}
which proves item~\ref{lb-point-lem1-2}.

We can represent $D_{\varepsilon}^{*}(g^{*};x_{0})$ as the maximal
convex minorant of $g$ defined by
$g= \min(g, g(x_{0})-\varepsilon+ \delta(\cdot|x_{0}))$.
For $x\in\dom g$, by Lemma S.A.10, %S.\ref{supp:cvx-subg-prop2},
$g^{*}(v_{0}) + g(x_{0})=\langle v_{0},x_{0}\rangle$.
Thus,
\[
\bigl(g(x_{0})-\varepsilon+ \delta(\cdot|x_{0})\bigr)^{*}(y)
= \langle x_{0}, y\rangle- g(x_{0}) + \varepsilon
= \langle x_{0}, y-v\rangle + \varepsilon
\]
for some $v\in\partial g(x_{0})$.
By Lemma S.A.7, %S.\ref{supp:cvx-mcm-min},
we have
\[
D_{\varepsilon}^{*}(g^{*};x_{0})^{*}
= \max(g^{*}, l_{0}^{*}),\qquad
l_{0}^{*}(y) \equiv\langle x_{0}, y-v\rangle + \varepsilon,
\]
which concludes the proof the lemma.
\end{pf}

%%--------------------------------------------------------------------------------------------------------------------
Since the domain of the quadratic part of equation \eqref
{lb-point-decomp} is
$\Real^{d}$, by Lem\-ma~S.A.11, %S.\ref{supp:cvx-subg-prop3},
we have that for
$x_{0}\in\dom g$ and
$v_{0}\in\partial g(x)$, there exists $w_{0}\in\partial q(x)$ such that
%
%e3.8 ###
\begin{equation}\label{lb-point-decomp-g}
v_{0} = G(x-x_{0}) + w_{0}.
\end{equation}
Therefore, for the point $x_{1}$ in the neighborhood $O(x_{0})$ where the
decomposition~\eqref{lb-point-decomp} is true, condition \eqref
{lb-supp-ineq-pos} is equivalent to
\[
\tfrac{1}{2}(x_{1}-x_{0})^{T}G (x_{1}-x_{0}) + q(x_{1})-\varepsilon\ge
\langle w_{0}, x_{1}-x_{0}\rangle+ q(x_{0}).
\]
Since $\langle w_{0}, x_{1}-x_{0}\rangle+ q(x_{0})$ is a support plane to
$q(x)$, the inequality~\eqref{lb-supp-ineq-pos} is satisfied if
$2^{-1}(x_{1}-x_{0})^{T}G (x_{1}-x_{0}) \ge\varepsilon$,
which is the complement of an open ellipsoid $B_{G}(x_{0},\sqrt
{2\varepsilon})$
defined by $G$ with center at $x_{0}$.
For $\varepsilon$ small enough, this ellipsoid will belong to the neighborhood
$O(x_{0})$. Since
$|D_{\varepsilon}(g;x_{0},v_{0})-g|\le\varepsilon$,
this proves that the family $D_{\varepsilon}(g;x_{0},v_{0})$ is a
local deformation.

In the same way, the condition~\eqref{lb-supp-ineq-neg} is equivalent to
\[
\tfrac{1}{2}(x_{1}-x_{0})^{T}G (x_{1}-x_{0}) + q(x_{1})
+\varepsilon\le\langle G(x_{1}-x_{0}) + w_{1}, x_{1}-x_{0}\rangle+ q(x_{0})
\]
or % \intertext{or:}
$2^{-1}(x_{1}-x_{0})^{T}G (x_{1}-x_{0}) + q(x_{0})-\varepsilon\ge
\langle w_{1}, x_{0}-x_{1}\rangle+ q(x_{1})$,
which is satisfied if we have
$2^{-1}(x_{1}-x_{0})^{T}G (x_{1}-x_{0}) \ge\varepsilon$.
Since $|D_{\varepsilon}^{*}(g;x_{0})-g|\le\varepsilon$,
this proves that the family $D_{\varepsilon}^{*}(g;x_{0})$ is also a
local deformation. Thus, we have proven the following.
%%--------------------------------------------------------------------------------------------------------------------
%
\begin{lem}
\label{lb-point-loc}
Let $g$ be a closed proper convex function, locally $G$-strongly convex
at some
$x_{0}\in\ri\dom g$ and $v_{0}\in\partial g(x_{0})$. The families
$D_{\varepsilon}(g;x_{0},v_{0})$ and $D_{\varepsilon}^{*}(g;x_{0})$ are
then local deformations for all $\varepsilon>0$ small enough.
Moreover, the condition
$2^{-1}(x-x_{0})^{T} G (x-x_{0}) \ge\varepsilon$
implies that
$D_{\varepsilon}(g;x_{0},v_{0})(x) = D_{\varepsilon}^{*}(g;x_{0})(x)
= g(x)$;
equivalently, $\supp[D_{\varepsilon}(g;x_{0},v_{0})-g]$ and
$\supp[D_{\varepsilon}^{*}(g;x_{0})-g]$ are subsets of
$B_{G}(x_{0},\sqrt{2\varepsilon})$.
\end{lem}
%
%%--------------------------------------------------------------------------------------------------------------------

For $r>0$ small enough, $h'\circ g(x)$ is nonzero and the
decomposition (\ref{lb-point-decomp}) is true on $B(x_{0};r)$.
Let us fix some $v_{0}\in\partial g(x_{0})$, some $x_{1}\in B(x_{0};r)$
such that $x_{1}\neq x_{0}$ and some $y_{1}\in\partial g(x_{1})$.
We fix $\delta$ such that equation~\eqref{lb-deflim3} of Lemma \ref
{lb-defnorm} is true for the transformation
$\sqrt{h}$ and $r=2$, and also $x_{0}\notin\overline
{B_{G}(x_{1};\sqrt{2\delta})}$.
Then, by Lemma~\ref{lb-point-loc}, for all $\varepsilon>0$ small enough,
the support sets $\supp[D_{\varepsilon}(g;x_{0},v_{0})-g]$ and $\supp
[D_{\delta}^{*}(g;x_{1})-g]$
do not intersect; that is, these two deformations do not interfere.

We can now prove Theorem~\ref{lb-point-main}.
The argument below is identical for $g^{+}_{\varepsilon}$ and~$g^{-}_{\varepsilon}$,
so we will give the proof only for $g^{+}_{\varepsilon}$.
We define deformations $g^{+}_{\varepsilon}$ and $g^{-}_{\varepsilon
}$ by means of the following lemma.
%%--------------------------------------------------------------------------------------------------------------------
%
\begin{lem}[(S.3.3)]%[S.\ref{supp:lb-point-exist}]
\label{lb-point-exist}
For all $\varepsilon>0$ small enough,
there exist
$\theta_{\varepsilon}^{+}, \theta_{\varepsilon}^{-}\in(0,1)$
such that the functions $g^{+}_{\varepsilon}$ and $g^{-}_{\varepsilon
}$ defined by
\begin{eqnarray*}
g^{+}_{\varepsilon} &=& (1-\theta_{\varepsilon}^{+})D_{\varepsilon
}(g;x_{0},v_{0}) + \theta_{\varepsilon}^{+}D^{*}_{\delta}(g;x_{1}),\\
g^{-}_{\varepsilon} &=& (1-\theta_{\varepsilon
}^{-})D^{*}_{\varepsilon}(g;x_{0}) + \theta_{\varepsilon
}^{-}D_{\delta}(g;x_{1};v_{1})
\end{eqnarray*}
belong to $\Model(h)$.
\end{lem}
%
%%--------------------------------------------------------------------------------------------------------------------

Next, we will show that $\theta_{\varepsilon}^{+}$ goes to zero fast
enough so that $g^{+}_{\varepsilon}$ is very close to $D_{\varepsilon
}(g;x_{0},v_{0})$.
Since supports do not intersect, we have
\begin{eqnarray*}
0 &=& \int(h \circ g^{+}_{\varepsilon} - h \circ g) \,dx \\
&=&\int\bigl(h \circ\bigl((1-\theta^{+}_{\varepsilon})D_{\varepsilon
}(g;x_{0},v_{0})+\theta^{+}_{\varepsilon} g\bigr)- h\circ g\bigr)\,dx\\
&&{} - \int\bigl( h \circ g - h \circ\bigl((1-\theta_{\varepsilon}^{+})g +
\theta_{\varepsilon}^{+} D^{*}_{\delta}(g;x_{1})\bigr)\bigr)\,dx,
\end{eqnarray*}
where both integrals have the same sign.
For the first integral, by Lemma~\ref{lb-deflim}, we have
\begin{eqnarray*}
&& \int\bigl|h \circ\bigl((1-\theta^{+}_{\varepsilon})D_{\varepsilon
}(g;x_{0},v_{0})+\theta^{+}_{\varepsilon} g\bigr)- h \circ g\bigr|\,dx \\
&&\qquad  \le
\int| h \circ D_{\varepsilon}(g;x_{0},v_{0})-h \circ g
|\,dx\\
&&\qquad  \asymp\int\bigl(g - D_{\varepsilon}(g;x_{0},v_{0})\bigr)\,dx
\le\varepsilon\mu\bigl[B_{G}\bigl(x_{0};\sqrt{2\varepsilon}\bigr)\bigr].
\end{eqnarray*}
The second integral is monotone in $\theta^{+}_{\varepsilon}$ and, by
Lemma~\ref{lb-defnorm}, we have
\[
\int\bigl(h \circ g - h \circ\bigl((1-\theta_{\varepsilon}^{+})g +
\theta_{\varepsilon}^{+} D^{*}_{\delta}(g;x_{1})\bigr)\bigr)\,dx
\asymp\theta_{\varepsilon}^{+}.
\]
Thus, we have $\theta_{\varepsilon}^{+}=O(\varepsilon^{1+d/2})$ and
\[
\lim_{\varepsilon\to0} \varepsilon^{-1}\bigl(g^{+}_{\varepsilon}(x_{0})
- g(x_{0})\bigr)
= \lim_{\varepsilon\to0} (1-\theta_{\varepsilon}^{+}) = 1.
\]

For Hellinger distance, we have
\begin{eqnarray*}
H(h \circ g^{+}_{\varepsilon}, h \circ g) &=&
H\bigl(h \circ\bigl((1-\theta_{\varepsilon}^{+})D_{\varepsilon
}(g;x_{0},v_{0}) + \theta_{\varepsilon}^{+}g\bigr), h \circ g\bigr)\\
&&{} + H\bigl(h \circ\bigl((1-\theta_{\varepsilon}^{+})g + \theta_{\varepsilon
}^{+}D^{*}_{\delta}(g;x_{1})\bigr), h\circ g\bigr).
\end{eqnarray*}
We can now apply Lemma~\ref{lb-deflim}:
\begin{eqnarray*}
H^{2}\bigl(h \circ\bigl((1-\theta_{\varepsilon}^{+})D_{\varepsilon
}(g;x_{0},v_{0}) + \theta_{\varepsilon}^{+}g\bigr), h \circ g\bigr)
&\le& H^{2}\bigl(h \circ D_{\varepsilon}(g;x_{0},v_{0}), h \circ g\bigr), \\
\lim_{\varepsilon\to0} \frac{H^{2}(h \circ D_{\varepsilon
}(g;x_{0},v_{0}), h \circ g)}
{\int(D_{\varepsilon}(g;x_{0},v_{0}) - g)^{2} \,dx}
&=& \frac{h'\circ g(x_{0})^{2}}{4h\circ g(x_{0})}
\end{eqnarray*}
and
\[
\int\bigl(D_{\varepsilon}(g;x_{0},v_{0}) - g\bigr)^{2} \,dx
\le\varepsilon^{2}\mu\bigl[B_{G}\bigl(x_{0};\sqrt{2\varepsilon}\bigr)\bigr]
= \varepsilon^{2+d/2} \frac{2^{d/2}\mu[S(0,1)]}{\sqrt{\det G}}.
\]
This yields
\begin{eqnarray*}
&&\limsup_{\varepsilon\to0}\varepsilon^{-({d+4})/{4}}
H \bigl(h \circ\bigl((1-\theta_{\varepsilon}^{+})D_{\varepsilon
}(g;x_{0},v_{0}) + \theta_{\varepsilon}^{+}g\bigr), h \circ g\bigr) \\
&&\qquad  \le C(d)\biggl(\frac{h '\circ g(x_{0})^{4}}{h \circ
g(x_{0})^{2}\det G}\biggr)^{1/4},
\end{eqnarray*}
where $S(0,1)$ is the $d$-dimensional sphere of radius $1$.

For the second part, by Lemma~\ref{lb-defnorm}, we obtain
\[
\limsup_{\varepsilon\to0} (\theta_{\varepsilon}^{+})^{-2}
H^{2}\bigl( h\circ\bigl((1-\theta_{\varepsilon}^{+})g + \theta_{\varepsilon
}^{+}D^{*}_{\delta}(g;x_{1})\bigr), h \circ g\bigr)
<\infty
\]
and
\[
H\bigl(h \circ\bigl((1-\theta_{\varepsilon}^{+})g + \theta_{\varepsilon
}^{+}D^{*}_{\delta}(g;x_{1})\bigr), h \circ g\bigr)
= O\bigl(\varepsilon^{({d+2})/{2}}\bigr).
\]
Thus,
\[
\limsup_{\varepsilon\to0} \varepsilon^{-({d+4})/{4}}
H( h \circ g^{+}_{\varepsilon}, h \circ g)\le C(d)\biggl(\frac
{h'\circ g(x_{0})^{4}}{h \circ g(x_{0})^{2}\det G}\biggr)^{1/4}.
\]
Finally, we apply Corollary~\ref{lb-cor}:
\[
\liminf_{n\to\infty} n^{{2}/({d+4})}
R_{1}(n;T,\{g,g_{n}\})\ge C(d)\biggl[\frac{h \circ g(x_{0})^{2}\det
G}{h'\circ g(x_{0})^{4}}\biggr]^{{1}/({d+4})}.
\]
Taking the supremum over all $G \in\mathcal{SC}(g;x_{0})$, we obtain
the statement of the theorem.
\end{pf*}

%s3.5 ###
\subsection{Indications of proofs for conjectured rates}

From
\citet{MR1240719} %Birg\'e and Massart (1993)
and
\citet{MR1385671}, %van der Vaart and Wellner (1996),
we expect that the global rate
of convergence of the MLE $\hat{p}_n$ of $p_0 = h \circ g_0$ in the
class $\mathcal{P} (h)$ will be determined by
the entropy of the class of convex and Lipschitz functions $g$ on
convex bounded domains
$C$, as given by
\citet{MR0415155} %Bronshtein (1976)
and
\citet{MR1720712}: % Dudley (1999):
if $\mathcal{F}_{L,C}$ is the class of all convex functions
defined on a compact convex set $C \subset\Real^d$ such that $| f(x)
- f(y) | \le L \| x - y \|$
for all $x,y \in C$, then the covering numbers for $\mathcal{F}_{L,C}$ satisfy
%
%e3.9 ###
\begin{equation}\label{CoveringBoundConvexLipschitz}
\log N(\epsilon, \mathcal{F}_{L,C}, \| \cdot\|_{\infty} ) \le
K(1+L)^{d/2} \epsilon^{-d/2}
\end{equation}
for all (small) $\epsilon>0$, for a constant $K$ depending only on $C$
and $d$.
Then, after an argument to transfer this covering number bound to a
bracketing entropy bound
for $\mathcal{P} (h)$ with respect to Hellinger distance $H$, it
follows from oscillation bounds for
empirical processes [cf. \citet{MR1385671}, %van der Vaart and
%Wellner (1996),
Theorems~3.4.1 and~3.4.4]
that rates of convergence of $\hat{p}_n$ with respect to Hellinger
distance are determined by
$r_n^2 \phi(1/r_n) \asymp\sqrt{n}$ with
%
%e3.10 ###
\begin{equation}
\phi(\delta) \equiv\int_{c \delta}^\delta\sqrt{1 + \log N_{[
]} (\epsilon, \mathcal{P} (h), H)} \,d \epsilon.
\end{equation}
Assuming that the bound of (\ref{CoveringBoundConvexLipschitz}) can be
carried over to
$\log N_{[   ]} (\epsilon, \mathcal{P} (h), H)$ sufficiently
closely, routine calculations show that the
expected rates of convergence of $\hat{p}_n$ to $p_0 = h(g_0)$ with
respect to Hellinger distance $H$
are
\[
r_n = \cases{
n^{2/(d+4)},  &\quad if $d \in\{ 1,2,3 \}$, \cr
\bigl(n/(\log n)^2\bigr)^{1/4},  &\quad if $d = 4$, \vspace*{2pt}\cr
n^{1/d} , &\quad if $d> 4$.}
\]
Based on these heuristics, we expect that the MLE $\hat{p}_n$ will be
rate efficient if $d \le3$,
but rate inefficient (not attaining the optimal rate $n^{2/(d+4)}$) if
$d \ge4$.

\subsubsection*{Some details}

\mbox{}

\textit{Case} 1: $ d \le3$. In this case, we find that
\begin{eqnarray*}
\phi(\delta)
& = & \int_{c\delta^2}^{\delta} \sqrt{1 + \log N_{[  ]} (\epsilon
, \mathcal{F}_{L,C}, \| \cdot\| )} \,d \epsilon\\
& \asymp& \int_{c \delta^2} ^\delta\sqrt{K (1+ L)^{d/2}} \epsilon
^{d/4} d\epsilon\\
& \asymp& M_1 \delta^{1-d/4},
\end{eqnarray*}
where $M_1 \equiv(K(1+L)^{d/2})^{1/2} / (1- d/4)$. Solving the relation
$r_n^2 \phi(1/r_n ) \asymp\sqrt{n}$ for $r_n$ yields $r_n =
n^{2/(d+4)}$ up to a constant.

\textit{Case} 2: $d=4$. In this case, we find that
\[
\phi(\delta) \asymp M_2 \log\biggl( \frac{1}{c \delta} \biggr),
\]
where $M_2 \equiv(K(1+L)^{d/2})^{1/2}$.
Solving the relation $r_n^2 \phi(1/r_n) \asymp\sqrt{n}$ for $r_n$ yields
$r_n = (n/(\log n)^2)^{1/4}$ up to a constant.

\textit{Case} 3: $d>4$. In this case, we calculate
\[
\phi(\delta) \asymp M_2 \delta^{2(1-d/4)},
\]
where $M_3 \equiv(K(1+L)^{d/2})^{1/2}/(d/4-1)$.
Solving the relation $r_n^2 \phi(1/r_n) \asymp\sqrt{n}$ for $r_n$ yields
$r_n = n^{1/d}$ up to a constant.

%%-------------------------------------------------------------------------------------------------------
%%-------------------------------------------------------------------------------------------------------

%V(x) &=& \prod_{k=1}^{d}x_{k},  x\in\Real_+^d & B(x_{0};r) &=&
%f(x_n)),  x_i\in\Real^d\\
%function}\}\\
%B_{H}(x_{0};r) &=& \{x : (x-x_{0})^{T}H(x-x_{0})<r^{2}\}\\
%$x$}\\

%----------------------------------------------------------
\section*{Acknowledgments}
%----------------------------------------------------------
This research is part of the Ph.D. dissertation of the
first author at the University of Washington.
We would like to thank two referees for a number of helpful suggestions.

%----------------------------------------------------------
%
\begin{supplement}[id=supp]
\sname{Supplement}
\stitle{Omitted Proofs and Some Facts from Convex Analysis}
\slink[doi]{10.1214/10-AOS840}
\sdatatype{.pdf}
\sfilename{ConvexTransfSupp-v4.pdf}
\sdescription{In the supplement, we provide omitted proofs and some
basic facts from convex analysis used in this paper.}
\end{supplement}

%suskaldyti doi

%
\printaddresses

\end{document}